\documentclass[11pt]{article}
\usepackage[]{graphicx}
\usepackage[]{color}
\makeatletter
\def\maxwidth{ %
  \ifdim\Gin@nat@width>\linewidth
    \linewidth
  \else
    \Gin@nat@width
  \fi
}
\makeatother

\usepackage{natbib}
\definecolor{fgcolor}{rgb}{0.345, 0.345, 0.345}

\usepackage{framed}
\makeatletter
 {\par\unskip\endMakeFramed%
 \at@end@of@kframe}
\makeatother

\definecolor{shadecolor}{rgb}{.97, .97, .97}
\definecolor{messagecolor}{rgb}{0, 0, 0}
\definecolor{warningcolor}{rgb}{1, 0, 1}
\definecolor{errorcolor}{rgb}{1, 0, 0}

\usepackage{alltt}
\usepackage{amsthm}
\usepackage[utf8]{inputenc}
\usepackage[english]{babel}
\newtheorem{theorem}{Theorem}
\usepackage{mathrsfs}
\usepackage{lipsum} 
\usepackage{listings,amsmath,amsthm,amssymb}
\usepackage{graphicx}
\usepackage[sc]{mathpazo} 
\usepackage[T1]{fontenc} 
\usepackage{microtype} 

\usepackage[hmarginratio=1:1,top=32mm,columnsep=20pt, left=1.05in, right=1.05in]{geometry} 
\usepackage{multicol} 
\usepackage[hang, small,labelfont=bf,up,textfont=it,up]{caption} 
\usepackage{booktabs} 
\usepackage{float} 
\usepackage{hyperref} 

\usepackage{lettrine} 
\usepackage{paralist} 

\usepackage{abstract} 
	
\usepackage[noabbrev]{cleveref}
\usepackage{fancyhdr}

\usepackage{titlesec} 
\usepackage{doi}
\usepackage{authblk}

\theoremstyle{condition}
\newtheorem{condition}{Condition}
\theoremstyle{definition}
\newtheorem{definition}{Definition}
\theoremstyle{remark}
\newtheorem{remark}{Remark}
\theoremstyle{corollary}
\newtheorem{corollary}{Corollary}
\theoremstyle{lemma}
\newtheorem{lemma}{Lemma}
\theoremstyle{proposition}
\newtheorem{proposition}{Proposition}
\theoremstyle{assumption}

\theoremstyle{comment}

\theoremstyle{acknowledgments}

\title{Minimax rates of convergence for nonparametric location-scale models} 
\date{}

\IfFileExists{upquote.sty}{\usepackage{upquote}}{}

\author[a]{Bingxin Zhao}
\author[a]{Yuhong Yang}
\affil[a]{School of Statistics, University of Minnesota, 313 Ford Hall, 224 Church St SE, Minneapolis, USA}

\begin{document}

\maketitle
\vspace{-1.5cm}

\begin{abstract}
	This paper studies minimax rates of convergence for nonparametric  location-scale models, which include mean, quantile and expectile regression settings. Under Hellinger differentiability on the error distribution and other mild conditions, we show that the minimax rate of convergence for estimating the regression function under the squared $L_2$ loss is determined by the metric entropy of the nonparametric function class. Different error distributions, including asymmetric Laplace distribution, asymmetric connected double truncated gamma distribution, connected normal-Laplace distribution, Cauchy distribution and asymmetric normal distribution are studied as examples. Applications on low order interaction models and multiple index models are also given.
\end{abstract}

$\textbf{\large Keywords:}$ quantile regression, expectile regression, minimax, metric entropy.

\newpage
\section{Introduction}
Nonparametric regression is a widely used tool to characterize relationship between a response variable and the predictors. In many economic and related applications, the random errors often exhibit heavy tail behaviors and regression methods under normal or other light-tail error distributions may lead to overly optimistic conclusions via the poorly estimated regression relationship. For instance, quantile and expectile regressions are frequently applied for estimation of Value-at-risk (VaR) and expected shortfall (ES), which are standard risk measures for financial institutions (see, e.g., \cite{Taylor}, \cite{Gaglianone}, \cite{Bayer},  and \cite{Yongqiao}). By definition, VaR and ES attempt to quantify extreme losses (in somewhat different ways). As such, they can be very sensitive to the tail behaviors of the error distribution. Obviously, overly optimistic VaR and ES estimates do not serve the intended purpose of preventing financial disasters. In this background, it is important to study nonparametric regression under heavy-tailed error distributions. Our goal in this work is to provide a theoretical understanding on intrinsically how well one can estimate a target regression function in a general nonparametric location-scale framework. Below, we briefly review quantile and expectile regressions, and discuss the lack of a general understanding in the existing literature of nonparametric regression from perspectives of minimax rates of convergence especially when the regression error has a heavy-tail distribution (e.g., Cauchy), which motivates our work. 

\subsection{Quantile regression}
Quantile regression (QR), initiated by \cite{RogerKoenker}, is naturally associated with an asymmetric absolute loss function $\rho_\tau(y)=y(\tau-I(y<0))$ (also known as check loss), where $0<\tau<1$ is the quantile level. With the quantile level ranging from 0 to 1, QR allows one to make inference on the entire conditional distribution,  which goes much beyond the conditional mean/median regression. 

Linear QR has been researched extensively. \cite{RogerKoenker}, \cite{Pollard} and \cite{Chernozhukov} considered the asymptotic behaviors of linear estimators. Penalized linear QR has also been investigated by \cite{LanWang} and \cite{Kato}. 

To alleviate restrictiveness of linear modeling,  semiparamtric and nonparametric models have been studied in the literature. \cite{XumingHe} and \cite{Lee} studied the partially linear additive model and investigated the asymptotic properties of the proposed estimators. \cite{Sherwood} considered nonconvex penalized estimators with the oracle property under high dimensional partially linear additive QR. 

Moreover, high dimensional additive QR has been examined by \cite{Horowitz2}, \cite{ShaogaoLv} and \cite{Kato}. Specifically, \cite{Horowitz2} proposed a nonparametric additive estimator with the oracle property. \cite{ShaogaoLv} proposed a two-fold LASSO-type regularized learning approach in sparse additive QR in a reproducing kernel Hilbert space (RKHS) framework. \cite{Kato} analyzed the group LASSO estimator under a high dimensional additive setting. \cite{TracyWu} studied the single-index QR, in which the unknown link function is estimated by local linear QR. \cite{KejiaShan} advocated model combination for QR and proposed the adaptive quantile regression by mixing (AQRM) to average parametric and/or nonparametric QR methods so as to achieve their best performance. 

\subsection{Expectile regression}
Expectile regression (ER), initiated by \cite{Aigner} and further developed by \cite{Newey}, also offers an approach to understanding the full range of conditional distribution of the response given the covariates. ER is based on minimizing asymmetrically weighted square residuals, using the asymmetric square loss $l_\phi(y)=y^2\lvert \phi-I(y<0)\rvert$ at the expectile level $\phi\in(0,1)$.

There are strengths and weaknesses of ER compared with QR as discussed in the literature. Since the asymmetric square loss (ASL) is differentiable everywhere while the check loss is nondifferentiable at zero, ER with the ASL	 can be computationally more advantageous than QR based on the check loss (\cite{Newey}). 

On the other hand, the ASL is sensitive to outliers, while the check loss is more robust. Also, expectiles lack an intuitive interpretation. 

Linear ER has been well studied. \cite{Newey} and \cite{Efron} pioneered the investigation of asymmetric least squares (ALS). \cite{YuwenGu} and \cite{QFXu} studied penalized linear ER. 

Without the restriction of linearity of the regression function, \cite{YaoQiwei}, \cite{Sobotka}, \cite{MengmengGuo}, \cite{Adam} and \cite{JunZhao2} investigated semiparametric or nonparametric ER. \cite{YiYang2} adopted the regression tree based gradient boosting estimator by \cite{Friedman}. \cite{YiYang} discussed the nonparametric multiple regression in the RKHS framework and proposed an estimator based on majorization-minimization (MM) algorithm. \cite{Spiegel2} extended the single index model (\cite{Ichimura}) to ER and proposed an algorithm to estimate the nonlinear function and covariates jointly. \cite{YuwenGu3} used the exponential weighting scheme to aggregate the different ER estimators and derived oracle inequalities for the aggregated estimator. \cite{CuixiaJiang} developed the expectile regression neural network (ERNN), which is estimated through standard gradient based optimization algorithms. 

\subsection{Existing minimax results on QR, ER and related problems}
The aforementioned asymptotic results on quantile and expectile regressions in the literature are mostly on risk upper bounds or pointwise asymptotic performance for specific QR and ER methods. Indeed, very little has been done on minimax QR and ER.

\cite{Horowitz} solved an "ill-posed" inverse problem of nonparametric instrumental variables estimator of QR, and derived an upper bound of optimal convergence rate in probability, which depends on the severity of the "ill-posed" inverse problem and the complexity of nonparametric estimator.  \cite{Kenko} considered functional QR, where the dependent variable is a scalar and the covariate is a function. He considered an estimator for the slope function by the principal component analysis, and derived an upper bound of minimax rate of convergence under some smoothness assumptions on the covariance kernel of the covariate and the slope function. \cite{XiaohongChen2} considered an "ill-posed" inverse problem under both nonparametric indirect regression (NPIR) and nonparametric instrumental variables regression (NPIV). They established an informative minimax lower bound under some approximation number and the link conditions. \cite{XiaohongChen} derived an upper bound for the convergence rate of a penalized sieve minimum distance (PSMD) estimator under the nonparametric additive quantile instrumental variable regression model, and showed the resulting convergence rate coincides with the known optimal rate. \cite{ShaogaoLv} proposed a two-fold LASSO-type regularized learning approach for sparse additive QR in a high dimensional RKHS framework. They derived an upper bound of convergence rate of the proposed estimator, which attains the minimax lower bound (up to a logarithmic factor in the ambient dimension) for sparse additive mean regression in \cite{Raskutti}.

\cite{Muhammad} explored the use of kernel-based regularized empirical risk minimization methods for ER. They considered the Gaussian kernels and ALS loss function, and derived an upper bound of convergence rate. Also, the established rate is minimax optimal up to a polynomial order of logarithmic factor of the sample size $n$.

In summary, while various interesting results have been obtained, there are no minimax results that hold for general quantile or expectile function classes. In fact, to our knowledge, no formal risk lower bounds have been established for ER and QR under a general non-normal error distribution. In particular, it is unknown if a severely heavy tailed distribution such as Cauchy would lead to a slower minimax rate of convergence for the estimation of a quantile function. The objective of this paper is to offer such a theory.

\subsection{Contribution and organization of this paper}
This work studies the minimax rate of convergence for nonparametric regression in a location-scale framework in terms of the metric entropy of the class of regression functions of interest. Metric entropy is known to be the key quality that determines the minimax rate of convergence in mean regression, see, e.g, \cite{Yang} for key results and references. The error distribution of QR and ER models typically does not follow normality as considered in, e.g, \cite{Yang} and \cite{XiaohongChen2}, and may have much heavier and asymmetric tails. The main contribution of our paper is that we derive the minimax rates of convergence precisely under the squared $L_2$ distance for a general class of regression functions, under mild conditions. In particular, our results are applicable for QR and ER with asymmetric Laplace distribution, asymmetric connected double truncated gamma distribution, connected normal-Laplace distribution, Cauchy distribution (for QR) and asymmetric normal distribution. 

Since metric entropy orders are known (or can be derived) for most of familiar function classes, our results readily enable the determination of minimax rates of convergence for QR and ER in the nonparametric location-scale framework. It is helpful to note that in previous work on convergence of QR and ER methods, minimax lower rates for mean regression under normal errors are typically used to justify optimality of the proposed methods (sometimes with gaps between upper and lower rates in logarithmic factors in the sample size). This does not seem to be a generally rigorous argument, especially for heavy-tail error distributions such as Cauchy that does not even have the first moment. Our work provides optimal minimax lower rates that are valid for general error distributions under minor conditions. Interestingly, it turns out that heavy-tailedness of the regression error does not damage the minimax rate of convergence for regression estimation. For example, for estimating a quantile function in a Lipschitz class on $[0,1]^d$ with smoothness parameter $\alpha >0$, the minimax rate of convergence under squared $L_2$ loss is $n^{-2\alpha/(2\alpha +d)}$, whether the error is Gaussian, double-exponential or Cauchy. Therefore, while the long tail of Cauchy is certainly likely to produce a significant number of outliers, they do not affect the rate of convergence if the regression function is estimated properly (see \Cref{non_normal_lower_rates} for more discussions).

The rest of this article is organized as follows. In \cref{problem_setup and some preliminaries}, we set up the problem with some useful definitions. In \cref{main_results}, we present the general results on minimax rates for location-scale models including nonparametric QR and ER. \Cref{error_qr} examines the examples of the asymmetric Laplace error distribution, the asymmetric connected double truncated gamma error distribution, the Cauchy distribution and the connected normal-Laplace error distribution. \Cref{error_er} considers the asymmetric normal error distribution for ER. In \cref{examples}, we provide applications on low order interaction models and multiple index models. \Cref{feasible_estimator} discusses several important issues and considers a feasible estimator with (near) minimax optimality. \Cref{conclusion} concludes our paper. The proofs of the main results are given in the \Cref{appendix} and Supplementary Materials.

\section{Problem setup}\label{problem_setup and some preliminaries}
In this paper, in a nonparametric location-scale regression framework, we unify mean, quantile, expectile regressions and their generalizations. In the following content, when comparing nonnegative sequences, we will use the symbols $\succeq$ and $\asymp$, where $k_n\succeq s_n$ is the same as $s_n=O(k_n)$ and $k_n\asymp s_n$ means $k_n\succeq s_n$ and $k_n\preceq s_n$.
\subsection{Location-scale models}
We assume the regression model:
\begin{equation}\label{eq:1}
	Y_i=\mu(X_i)+\sigma(X_i)\epsilon_i, i=1,2,\cdots,n,
\end{equation}
where $Y_i$ is a real valued response random variable, $\mu(x)$ is the regression function, $\sigma (x)$ is a scale function that, together with $\epsilon$, determine the heterogeneous random noise. Here $\epsilon_i$ has a distribution with known density $f$ with respect to the Lebesgue measure, which may have a heavy tail (e.g., without a finite mean), and the covariate $X_i$ and $\epsilon_i$ are independent. The observations $(X_i,Y_i)_{i=1}^n$ are i.i.d. from a distribution with a generic copy $(X, Y)$, and the explanatory variable $X\in {\cal{X}}$ has distribution $P$ with density $h$ belonging to a density class $\mathcal{H}$. The regression function $\mu$ is assumed to be in a nonparametric function class $\mathcal{U}$, where $\mathcal{U}$ satisfies $\underset{\mu\in \mathcal{U}}{\sup}\lVert \mu\rVert_\infty<\infty$. The scale function $\sigma$ is assumed to be in some function class $\Sigma$ with lower and upper bounds: there exists a fixed $\overline{\sigma}>1$ and $\underline{\sigma}=\overline{\sigma}^{-1}$ such that $\underline{\sigma}\le \sigma(x)\le \overline{\sigma}$ for all $x\in \cal{X}$. For example, $\Sigma = \{ \exp{\{\zeta (x; \nu)}\}: \nu\in \Psi\}$, where $\zeta (x; \nu)$ is of a parametric form (e.g., linear in the covariates) and  $\Psi$ is a compact parameter space. 

We call  \eqref{eq:1} a location-scale regression model. It specializes to different types of regression depending on the property of $f$: e.g., if $f$ has mean 0, then \eqref{eq:1} is a usual mean regression model; if $f$ has $\tau$-th quantile equal 0, then \eqref{eq:1} corresponds to a QR model at level $\tau$; if $f$ has $\tau$-th expectile equal 0, then \eqref{eq:1} is an ER model. More details are given below.

\subsection{Mean regression}
If the error $\epsilon$ has mean 0, the location-scale model \eqref{eq:1} corresponds to a mean regression model with heteroscedastic error. Here the regression function is the conditional mean of $Y$ given $X=x$:
\begin{equation*}
	\mu(x)=\mathbb E(Y|X=x).
\end{equation*}
	Note that $\mu(x)$ is well defined if $\mathbb{E}(|Y||X=x)$ exists almost surely. However, if an estimator is obtained by least squares, a finite second moment of $Y$ given $X=x$ is typically required.

\subsection{Quantile regression}

If the error $\epsilon$ has a distribution with the $\tau$-th quantile being 0,  \eqref{eq:1} is just a QR model. The regression function $\mu (x)=\mu_\tau (x)$ is the $\tau$-th conditional quantile of $Y$ given $X=x$, which is defined as
\begin{equation*}
	\mu(x)=F_{Y}^{-1}(\tau|x)=\inf\left\{ y:F_{Y}(y|X=x)\geq\tau\right\},
\end{equation*}
where $F_{Y}(y|x)$ denotes the cumulative distribution function (CDF) of $Y$ given $X=x$.
Note that $\mu(x)$ is defined based on the inverse of the CDF, without any requirement of moment conditions of $Y|X=x$. However, when using the check loss $\rho_\tau(y)=y(\tau-I(y<0))$ to obtain a QR estimator, the existence of $E(|Y| |X=x)$ is required.

\subsection{Expectile regression}
If the error $\epsilon$ has a distribution with the $\phi$-th expectile equal 0 ($0<\phi<1$), \eqref{eq:1} corresponds to an ER model. The expectile regression function is the $\phi$-th conditional expectile of $Y$ given $X=x$:
\begin{equation*}
	\mu(x)=\left\{z:\int_{-\infty}^{z}(z-y)dF_{Y}(y|X=x)=\phi \int_{-\infty}^\infty\lvert z-y\rvert dF_{Y}(y|X=x)\right\}.
\end{equation*}
When $\phi=1/2$, expectile regression reduces to the mean regression.
Note that $\mu(x)$ is well defined if $\mathbb{E}(|Y||X=x)<\infty$. However, a finite second moment of $Y$ given $X=x$ is required in order to obtain the ER estimator using asymmetric least squares (ALS) based on the asymmetric square loss $l_\phi(y)=y^2\lvert \phi-I(y<0)\rvert$.

\subsection{Momentile and $\psi$-tile regressions} \label{psi-tile}
Expectile regression can be generalized to momentile regression. Let $k\geq 1$ be an integer and assume the $k$-th moment of the conditional distribution of $Y$ given $X=x$ exists for $x\in \cal{X}$. Then the $k$-th momentile regression function is the $\phi$-th conditional momentile of $Y$ given $X=x$ as defined by
\begin{equation*}
	\mu(x)= \left\{z:\int_{-\infty}^{z}(z-y)^kdF_{Y}(y|X=x)=\phi \int_{-\infty}^\infty \lvert z-y\rvert^k dF_{Y}(y|X=x)\right\}.
\end{equation*}
For $k$-momentile regression at level $\phi$, the error $\epsilon$ in (\ref{eq:1}) is assumed to have the $\phi$-th $k$-momentile equal 0. Similar to the case of expectile, when existing, there is a unique choice of $z$ that satisfies the integral equation above.  

The concept can be further generalized to facilitate e.g., robust regression under the Huber loss (to be discussed soon). Let $\psi: [0, \infty )\rightarrow [0, \infty )$ be a non-negative function satisfying $\int_{-\infty}^{\infty} \psi(|t-a|)f(t)dt<\infty$ for all $a\in \mathbb{R}$.  Then the $\psi$-tile regression function is the $\phi$-th conditional $\psi$-tile of $Y$ given $X=x$ as defined by
\begin{equation*}
	\mu_\phi (x;\psi)=\inf \left\{z:\int_{-\infty}^{z} \psi(z-y)dF_{Y}(y|X=x)=\phi \int_{-\infty}^\infty \psi(\lvert z-y\rvert) dF_{Y}(y|X=x)\right\}.
\end{equation*}
For $\psi$-tile regression at level $\phi$, the error $\epsilon$ is assumed to have the $\phi$-th $\psi$-tile equal 0.
Note that when $\psi=1$, $\psi$-tile regression reduces to the quantile regression. Similar to the case of quantile, the set of $z$ satisfying the above integral equality may not be unique and we choose their infimum for $\psi$-tile to be well-defined. For convenience, we will use $\psi$-tile regression as a general term to include quantile regression, mean regression, expectile regression, or regression under another sensible choice of $\psi$.

A non-standard choice of $\psi$ than those for the familiar quantile and expectile regressions can be useful. For instance, let $\psi_H(x)=2\min (x,1)$, which corresponds to the (differentiable) Huber loss function $x^2I_{\{0\leq x\leq 1\}}+(2x-1)I_{\{x>1\}}$. Then the resulting Huber-tile regression emphasizes robustness to outliers compared to the expectile regression while maintaining the computational advantage of expectile over quantile regression. If the distribution of $\epsilon$ is symmetric about 0, then the $1/2$-th Huber-tile of $Y$ given $x$ is simply the same as the median (and mean, when existing) regression function. But Huber-tile regression is oriented towards robustness to outliers. 

The $\psi$-tile regression at level $\phi$ corresponds to the loss function $L(y,\hat{y})=J(y-\hat{y})$, where $J(z)=\phi \int_0^z \psi(t)dt I_{\{z\geq 0\}} +(1-\phi) \int_0^{|z|} \psi(t)dt I_{\{z< 0\}}$. It yields the check loss and the asymmetric square loss with $\psi=1$ and $\psi(x)=x$ respectively. For Huber-tile $\psi_H$, we have $J(z)=\phi \big(z^2I_{\{0\le z\le 1\}}+(2z-1)I_{\{z>1\}})\big)+(1-\phi)\big(z^2I_{\{-1\le z< 0\}}-(2z+1)I_{\{z<-1\}}\big)$.

Based on model (\ref{eq:1}) with a general error distribution of density $f$, the conditional $\psi$-tile of $Y$ given $X=x$, when existing, is of the form
\begin{equation} 
\mu(x)+\sigma(x) t_{\phi, \psi}(f), \label{psi-tile}
\end{equation}
where $t_{\phi, \psi}(f)$ denotes $\phi$-th $\psi$-tile of the error distribution with density $f$. Therefore, in the framework of location-scale model, the models for quantile, expectile, momentile, or $\psi$-tile, when existing, are closely related in terms of a simple shifting. For instance, if we assume model (\ref{eq:1}) with $\tau$-th quantile of the distribution of $\epsilon$ being zero, then the $\tau$-th conditional quantile function of $Y$ given $X=x$ is just $\mu(x)$, but the  $\phi$-th conditional $\psi$-tile is $\mu(x)+\sigma(x) t_{\phi, \psi}(f)$. 

\subsection{Minimax risk and metric entropy}
Let $\left\|{\mu-\upsilon}\right\|^2_{2,h}=\int (\mu(x)-\upsilon(x))^2h(x)dx$ be the squared $L_2(h)$ distance with respect to the measure induced by X. To assess estimation performance, we use the squared $L_2$ loss.

\begin{definition}
(Minimax risk) The minimax risk of estimating $\mu \in \mathcal{U}$ for the location-scale model under the squared $L_2$ loss is
\begin{equation*}
R_n=\underset{\hat{\mu}}{\inf} \underset{\mu\in \mathcal{U},\sigma\in\Sigma, h\in \mathcal{H}}{\sup}\mathbb{E}\lVert\mu-\hat{\mu}\rVert_{2,h}^2,
\end{equation*}
where $\hat\mu$ is over all possible estimators of $\mu$ based on the data.
\end{definition}

It is well-known that metric entropy plays a fundamental role in determining the minimax rate of convergence in statistical models (e.g, \cite{LeCam}, \cite{Birge}, \cite{Yatracos} and \cite{Yang} for results and more historical references). Consider a totally bounded function class $\mathcal{U}$ under the $L_2(h)$ distance.

\begin{definition}
(Packing $\varepsilon$-entropy) A finite set $N_\varepsilon \subset \mathcal{U}$ is said to be an $\varepsilon$-packing set in $\mathcal{U}$ with separation $\varepsilon>0$, if for any functions $\mu,\mu^\prime \in N_\varepsilon, \mu \ne \mu^\prime$, we have $\left\|{\mu-\mu^\prime}\right\|_{2,h}>\varepsilon$. The logarithm of the maximum cardinality of $\varepsilon$-packing sets is called the packing $\varepsilon$-entropy of $\mathcal{U}$ in the $L_2(h)$ distance, and is denoted $M(\varepsilon;\mathcal{U})$.
\end{definition}

In this paper, we are interested in nonparametric regression under location-scale model. For a nonparametric class $\mathcal{U}$, it usually satisfies the following richness property.
\begin{definition}
(Richness) A function class $\mathcal{U}$ has the richness property, if for some $0<\alpha<1$, it satisfies
\begin{equation*}
\begin{aligned}
\lim \underset{\varepsilon \to 0}{\inf} \frac{M(\alpha \varepsilon)}{M(\varepsilon)}>1,
\end{aligned}
\end{equation*}
where $M(\varepsilon)$ is a presumed known order of $M(\varepsilon;\mathcal{U})$.
\end{definition}

Note that for nonparametric function classes, the metric entropy order is typically $(\frac{1}{\varepsilon})^{c_1}(\log \frac{1}{\varepsilon})^{c_2}$ with $c_1>0$ and $c_2\in \mathbb{R}$. It clearly satisfied the richness condition.

\begin{definition}
	(Order regularity) A function class is (metric entropy) order regular if for every constant $c>0$, $M(\varepsilon)$ and $M(c\varepsilon)$ are of the same order, i.e.,
	\begin{equation*}
		M(\varepsilon) \asymp M(c\varepsilon), \text{as~} \varepsilon\to 0,
	\end{equation*}
	where $M(\varepsilon)$ is a presumed known order of $M(\varepsilon;\mathcal{U})$.
\end{definition}

Recall the definitions of Kullback-Leibler (K-L) divergence and Hellinger distance between two probability densities $q_1(x)$ and $q_2(x)$ with respect to a $\sigma$-finite measure $\nu$:  $D(q_1||q_2)=\int q_1(x)\log\frac{q_1(x)}{q_2(x)}d\nu$ and $d_H(q_1,q_2)=\Big (\int(\sqrt{q_1(x)}-\sqrt{q_2(x)})^2d\nu \Big)^{\frac{1}{2}}$, respectively.

\section{Main results}\label{main_results}
Recall $f(y)$ is the probability density function of the error $\epsilon_i$ with respect to the Lebesgue measure in the location-scale model \eqref{eq:1}. Clearly, conditions on $f$ are needed to derive the minimax rate of convergence for estimating $\mu$. Consider the location-scale family
\begin{equation}\label{loc_sca}
	\left\{ \frac{1}{\sigma}f\left(\frac{y-\eta}{\sigma}\right),\eta\in \mathbb{R}, \sigma\in (0,+\infty)\right\},
\end{equation}
and let $f_{\eta,\sigma}$ denote the density. We assume $D(f_{0,1}||f_{\eta,\sigma})<\infty$ for $\eta\in \mathbb{R}$ and $\sigma\in (0,\infty)$ and it is continuous in $(\eta,\sigma)$.

\begin{condition}\label{con1}
	The K-L divergence between $f_{0,1}$ and $f_{\eta,\sigma}$ is locally upper bounded in order by $\eta^2+(1-\sigma)^2$ for $(\eta,\sigma)$ near $(0,1)$. That is,
    	\begin{equation*}
    		D(f_{0,1}||f_{\eta,\sigma})= O(\eta^2+(1-\sigma)^2),
    	\end{equation*}
    	for $(\eta,\sigma)$ near $(0,1)$.
\end{condition}
\begin{remark}
\Cref{con1} is rather mild and it is satisfied by Gaussian, double-exponential, Student's t and many other distributions, see \cite{Yang3}, \cite{GangCheng} and \cite{GangCheng2}. 
\end{remark}

Note that if $D(f_{0,1}||f_{\eta,\sigma})<\infty$ is twice differentiable at $(\eta,\sigma)=(0,1)$, then \Cref{con1} is readily seen to be satisfied.

Hellinger differentiability plays an important role in our derivation of sensible minimax upper bounds. Let $\mathcal F=\{f_\theta:\theta\in \Theta\}$ be a subset of nonnegative functions that are integrable, indexed by $\theta$ in a subset $\Theta$ of $\mathbb R^2$. In our case, it is said to be Hellinger differentiable (see \cite{Hajek}, \cite{LeCam2} and \cite{Pollard2}) at a point $\theta_0$ of $\Theta$ if the map $\theta\to \xi_\theta(y):=\sqrt{f_\theta(y)}$ is differentiable in $L_2$ norm at $\theta_0$. That is, if there exists a vector $\dot{\xi}_{\theta_0}(y)$ of square integrable functions, called the Hellinger derivative at $\theta_0$, such that
\begin{equation}
\label{2}
    \xi_{\theta}(y)=\xi_{\theta_0}(y)+(\theta-\theta_0)^\prime\dot{\xi}_{\theta_0}(y)+r_{\theta}(y),
\end{equation}
where $\left\|r_{\theta}\right\|=\sqrt{\int r_{\theta}^2(y)dy}=o(\lvert \theta-\theta_0 \rvert)$ near $\theta_0$ and $\lvert \theta-\theta_0 \rvert$ is the Euclidean distance between $\theta$ and $\theta_0$.

In particular, we say the location-scale family \eqref{loc_sca} is Hellinger differentiable at $(\eta,\sigma)=(0,1)$ if the map $(\eta,\sigma)\to \xi_{\eta,\sigma}(y):=\sqrt{f_{\eta,\sigma}(y)}$ satisfies \eqref{2} at $(\eta,\sigma)=(0,1)$. 

\begin{condition}\label{con2}
	The location-scale family \eqref{loc_sca} is Hellinger differentiable at $(\eta,\sigma)=(0,1)$ and the components of the Hellinger derivative are not linearly dependent.
\end{condition}

Besides conditions on the error distribution, we also need some constraint on the density function class $\mathcal{H}$.
\begin{condition}\label{con4}
	There exists a density $h_0\in \mathcal{H}$ and a constant $C<\infty$ such that for all $h\in \mathcal{H}$, we have
	\begin{equation*}
		\sup_{x\in \mathcal{X}}\frac{h(x)}{h_0(x)}\le C.
	\end{equation*}
\end{condition}
Let $M_{h_0}(\varepsilon;\mathcal{U})$ and $M_{h_0}(\varepsilon;\Sigma)$ denote the metric entropy of $\mathcal{U}$ and $\Sigma$ under the $L_2(h_0)$ distance, respectively.
\begin{condition}\label{con3}
	 $M_{h_0}(\varepsilon;\mathcal{U})$ is rich and regular. In addition, $M_{h_0}(\varepsilon;\Sigma)=O(M_{h_0}(\varepsilon;\mathcal{U}))$ as $\varepsilon\to 0$.
\end{condition}

Condition \ref{con3} is sensible. In real applications, often parametric families of  $\Sigma$ are considered. For instance, for the family $\Sigma = \{ \exp{\{\zeta (x; \nu)}\}: \nu\in \Psi\}$, where $\zeta (x; \nu)$ is of a parametric form and  $\Psi$ is a compact parameter space, the metric entropy of  $\Sigma$ is of order $\log (1/\varepsilon)$, which is dominated by the metric entropy order of a nonparametric class $\mathcal{U}$. See \Cref{scale_function_discussion} for more discussion on the requirement $M_{h_0}(\varepsilon;\Sigma)=O(M_{h_0}(\varepsilon;\mathcal{U}))$.

\begin{theorem}\label{th1}
	For the location-scale model, under Conditions 1-4, the minimax rate of convergence of estimating the regression function $\mu\in\mathcal{U}$ is
	\begin{equation*}
    \begin{aligned}
    \underset{\hat{\mu}}{\inf}~ \underset{\mu\in\mathcal{U},\sigma\in \Sigma, h\in \mathcal{H}}{\sup}\mathbb{E}\left\|{\mu-\hat{\mu}}\right\|_{2,h}^2 \asymp \varepsilon_n^2,
    \end{aligned}
    \end{equation*}
    where $\varepsilon_n$ is determined by $ M_{h_0}(\varepsilon_n;\mathcal{U})=n\varepsilon_n^2$.
\end{theorem}

\begin{remark}

In the derivation of the minimax rate of convergence, as seen in the proof of the theorem, matching (in order) minimax upper and lower bounds are obtained, both in terms of the metric entropy of the regression class.

\end{remark}

\begin{remark}

In this paper, we focus on nonparametric location-scale model. Many results on parametric quantile/expectile regression have been obtained as mentioned in the introduction, with the familiar parametric rate of convergence. To show such a rate is minimax optimal, a local entropy argument is needed in the derivation of the minimax lower bounds via Fano's inequality (see \cite{Yang} and \cite{WangZhan}).  

\end{remark}

\begin{remark}

In the above theorem, we assume the density $f$ of $\epsilon$ is known. If $f$ is unknown, but known in a finite set of densities that satisfy Conditions 1-2, then the minimax rate of convergence stays the same. It is interesting to observe that in the literature,  upper bound results are often obtained with moment type conditions on the error distribution (e.g., \cite{ChenXiaohong}). While the density $f$ is not required to be known, such results work with specific smoothness function classes to facilitate the construction of an optimal-rate estimator. In contrast, our approach captures the essential role of the metric entropy of the regression function class at the expense of restricting $f$ to be known. 

\end{remark}

See \Cref{appendix} for the proof of \Cref{th1}. Note that for the commonly encountered smoothness function classes (e.g., Sobolev and Besov classes), the metric entropy orders are well-known (see, e.g., \cite{Lorentz}; \cite{Yang} for examples). Therefore the above theorem immediately characterizes the optimal rate of convergence for estimating a quantile/expectile (or $\psi$-tile) regression function in such a class. It is also important to point out that the conditions required for the theorem do not force the existence of the second or even the first moment of the error distribution as is often required for risk upper bounds on regression estimation.  To be fair, previous results in the literature may have the advantage of allowing general distributions of possibly dependent errors (e.g., \cite{XiaohongChen}).  It seems that in the literature, minimax lower bounds for regression are typically derived under normality. Therefore the lower bounding part of Theorem \ref{th1}   
provides a much needed understanding of the fastest possible rate of convergence as the right target when examining specific regression methods on their optimality in estimating the regression function where normality is inappropriate.  

For example, consider Lipschitz class $Lip(\alpha; C; [0,1]^d)$ that consists of all functions on $[0,1]^d$ with all $\lfloor \alpha \rfloor$ partial derivatives satisfying $\sup_{x_1,x_2\in [0,1]^d} |g(x_1)-g(x_2)|\leq C||x_1-x_2||^{\alpha-\lfloor \alpha \rfloor}$. Then the packing entropy of the class under the $L_2$ distance is known to be of order $\epsilon ^{-d/\alpha}$ as $\epsilon\rightarrow 0$. Besides the normal distribution, as will be seen in the next sections, asymmetric Laplace distributions, and Cauchy distributions all satisfy the conditions for the theorem. Therefore, solving the simple equation $\epsilon ^{-d/\alpha}=n\epsilon^2$, which gives the solution $\epsilon_n=n^{-\alpha/(2\alpha+d)}$, by Theorem 1, we know the minimax rate of convergence under the square $L_2$ loss is  $n^{-2\alpha/(2\alpha+d)}$ for the aforementioned error distributions, regardless of the degree of heavy-tailness.

\section{Error distribution examples of quantile regression}\label{error_qr}
\subsection{Asymmetric Laplace error distribution}

The asymmetric Laplace distribution (ALD) has been widely used in economics, finance and related applications (\cite{Hinkley-Revankar, Madan-Seneta}). More recently, \cite{QianChen} and \cite{Taylor2} investigated on the VaR and ES estimation based on ALD. It has the following probability density function:
\begin{equation*}
f(y|\eta,\sigma,\tau)=\frac{\tau(1-\tau)}{\sigma}\exp\Big[{-\rho_{\tau}\Big(\frac{y-\eta}{\sigma}\Big)}\Big],-\infty<y<\infty,
\end{equation*}
where $0<\tau<1$ is the skewness parameter, $\sigma>0$ is the scale parameter, $-\infty<\eta<\infty$ is the location parameter, and $\rho_\tau$ is the check loss. Also, the $\tau$-th quantile of the above ALD equals 0 if the location parameter $\eta=0$.

\subsection{Asymmetric connected double truncated gamma error distribution}

Although ALD has been widely used in asymmetric distribution scenarios, the peaky behavior at the origin is not desirable in many applications. To eliminate this behavior, \cite{Nassiri} constructed an asymmetric connected double truncated gamma (ACDTG) distribution (see \cite{Gijbels2019} for related properties of asymmetric densities). 

For $t>0$ and $k\in \mathbb{R}$, define the upper incomplete gamma function by
\begin{equation*}
	\Gamma(t,k)=\int_k^{+\infty}x^{t-1}e^{-x}dx.
\end{equation*}

The ACDTG distribution has the following probability density function:
\begin{equation*}
	f(y|\eta,\sigma,\alpha,\tau)=\frac{\tau(1-\tau)}{\sigma\Gamma(\alpha+1,\alpha)}[\alpha+\rho_\tau(\frac{y-\eta}{\sigma})]^\alpha \exp\{-[\alpha+\rho_\tau(\frac{y-\eta}{\sigma})]\},-\infty<y<\infty,
\end{equation*}
where $\alpha\ge 0$ is the shape parameter, $0<\tau<1$ is the skewness parameter, $\sigma$ is the scale parameter, and $-\infty<\eta<\infty$ is the location parameter. Also, the $\tau$-th quantile of ACDTG is 0 if the location parameter $\eta=0$.

\subsection{Connected normal-Laplace error distribution}
The previous asymmetric distributions have the same form in the density expression but different slope parameters on the two sides of $\eta$. It is conceivable that the asymmetry can go beyond. Here, we consider an example where the shapes of the density are different on the left and right sides. Specifically, the left side follows a normal distribution and the right side follows a Laplace distribution. We name it connected normal-Laplace (CNL) error distribution. The probability density function $f(y|\eta,\sigma,\tau,p)$  is 
\begin{equation*}
	\frac{2\tau}{\sqrt{2\pi}\sigma}\exp[-\frac{(y-\eta)^2}{2\sigma^2}]I(y\le \eta)+\frac{1-\tau}{\beta\sigma}\exp[-\frac{1}{\beta\sigma}(y-\eta)]I(y>\eta),-\infty<y<\infty,
\end{equation*}
where $\sigma>0$ is the global scale parameter, $\beta>0$ is the additional scale parameter for the right side of the distribution, $-\infty<\eta<\infty$ is the location parameter and $0<\tau<1$ is the proportion parameter. Also, we have the constraint $\frac{1-\tau}{\beta}=\frac{2\tau}{\sqrt{2\pi}}$, which guarantees that \Cref{con1} is satisfied. This distribution has $\tau$-th quantile being 0 if the location parameter $\eta=0$.

\subsection{Cauchy distribution}
The above distributions have sub-exponential behaviors. However,  heavier tails could be more proper in some applications. Thus, we consider an example with heavier tails than the previous distributions, which is the Cauchy distribution. It has the probability density function:

\begin{equation*}
	f(y|\eta, \sigma) = \frac{1}{\pi \sigma[1+(\frac{y-\eta}{\sigma})^2]}, -\infty < y < \infty,
\end{equation*}
where $-\infty < \eta <\infty$ is the location parameter and $\sigma>0$ is the scale parameter. 
\begin{remark}
	Since the mean of Cauchy distribution does not exist, the estimation of the quantile function cannot be realized through the use of the check loss. However, the quantile regression function can still be estimated optimally by other methods. 
\end{remark}

\subsection{Minimax rates under the above error distributions}
\begin{corollary}\label{cor1}
	Under Conditions 3-4, if the error of quantile regression follows one of the distributions presented in the previous subsections, the minimax rate of convergence of estimating the quantile regression function $\mu\in\mathcal{U}$ under the squared $L_2$ distance is
	\begin{equation*}
    \begin{aligned}
    \underset{\hat{\mu}}{\inf}~ \underset{\mu\in\mathcal{U},\sigma\in \Sigma,h\in\mathcal{H}}{\sup}\mathbb{E}\left\|{\mu-\hat{\mu}}\right\|_{2,h}^2 \asymp \varepsilon_n^2,
    \end{aligned}
    \end{equation*}
    where $\varepsilon_n$ is determined by $ M_{h_0}(\varepsilon_n;\mathcal{U})=n\varepsilon_n^2$.
\end{corollary}


\section{An error distribution example of expectile regression}\label{error_er}
\subsection{Asymmetric normal distribution}
Let us study a well-known distribution called asymmetric normal distribution, see \cite{OHagan} and \cite{Azzalini}. Some existing results on this distribution and the application on economics data can be found in \cite{JijiXing}. It has probability density function $f(y|\eta,\sigma,\phi)$:
\begin{equation*}
	\frac{2}{\sqrt{\pi}\sigma}\frac{\sqrt{\phi(1-\phi)}}{\sqrt{\phi}+\sqrt{1-\phi}}\exp[-\frac{\phi}{\sigma^2}(y-\eta)^2I(y\ge \eta)-\frac{1-\phi}{\sigma^2}(y-\eta)^2I(y<\eta)], -\infty<y<\infty,
\end{equation*}
where $0<\phi<1$ is the skewness parameter, $\sigma>0$ is the scale parameter, and $-\infty < \eta<\infty$ is the location parameter. Given the location parameter $\eta=0$, the $\phi$-th expectile equals 0.

\subsection{Minimax rate under the asymmetric normal distribution}
\begin{corollary}\label{cor2}
	Under Conditions 3-4, if the error of expectile regression follows the asymmetric normal distribution, the minimax rate of convergence of estimating the expectile regression function $\mu\in\mathcal{U}$ under the squared $L_2$ distance is
	\begin{equation*}
    \begin{aligned}
    \underset{\hat{\mu}}{\inf}~ \underset{\mu\in\mathcal{U},\sigma\in \Sigma,h\in\mathcal{H}}{\sup}\mathbb{E}\left\|{\mu-\hat{\mu}}\right\|_{2,h}^2 \asymp \varepsilon_n^2,
    \end{aligned}
    \end{equation*}
    where $\varepsilon_n$ is determined by $ M_{h_0}(\varepsilon_n;\mathcal{U})=n\varepsilon_n^2$.
\end{corollary}

\begin{remark}
	The expectiles (when existing) and quantiles are related to each other. In the location-scale model, \cite{YaoQiwei} gave a mapping between quantiles and expectiles. Hence, under the same distribution of $\epsilon$, the $\tau$-th quantile is equivalent to the $\phi$-th expectile at a corresponding level $\phi$ that depends on $\tau$ and the distribution of $\epsilon$. Therefore, except the Cauchy distribution, the error distributions in Section 4 can be used under ER and the error distribution in Section 5 can also be used under QR.
\end{remark}

\section{Minimax rates of convergence for some dimension reduction models}\label{examples}
In this section, we apply the main theorem to derive minimax rates for some dimension reduction models. In the following two subsections, we consider $\pmb x\in [0,1]^d$ and let $\mathcal{H}$ be the class of densities $h$ on $[0,1]^d$ that satisfied $1/\overline{C}\le h(x)\le \overline{C}$ for a constant $\overline{C}>1$. Suppose $\Sigma=\{\exp(\sum_{i=1}^dc_ix_i):\underset{1\le i\le d}{\max}\lvert c_i\rvert\le C\}$ for some known $C>0$. Let the error distribution be any of the examples considered in the previous sections.

\subsection{Additive and low order interaction models}
We first consider the regression functions in Sobolev classes with different interaction orders and smoothness.

For $r\ge 1$, let $\textbf{z}=(z_1,\cdots,z_r)\in[0,1]^r$ and $\textbf{k}=(k_1,\cdots,k_r)$ with integer $k_i\ge 0, i=1,\cdots,r$. Also, let $D^{\textbf{k}}$ denote the differentiation operator $D^{\textbf{k}}=\frac{\partial^{\lVert \textbf{k}\rVert_1}}{\partial z_1^{k_1}\cdots\partial z_r^{k_r}}$. For integer $\alpha$, define the $r$-dimensional Sobolev norm 
\begin{equation*}
	\lVert g\rVert_{\mathcal{G}_2^{\alpha,r}}=\lVert g\rVert_2+\left( \sum_{\lVert \textbf{k}\rVert_1=\alpha}\int_{[0,1]^r}\lvert D^{\textbf{k}}g\rvert^2d\textbf{z}\right)^{1/2}.
\end{equation*}
Let $\mathcal{G}_2^{\alpha,r}(C)$ denote the set of all functions $g$ on $[0,1]^r$ with $\lVert g\rVert_{\mathcal{G}_2^{\alpha,r}}\le C$ for some known $C>0$. Then, we consider the following function classes on $[0,1]^d$:

\begin{equation*}
\begin{aligned}
    &T_1(\alpha;C)=\{\sum_{i=1}^dg_i(x_i):g_i\in \mathcal{G}_2^{\alpha,1}(C), 1\le i\le d\},\\
	&T_2(\alpha;C)=\{\sum_{1\le i<j\le d}g_{i,j}(x_i,x_j):g_{i,j}\in \mathcal{G}_2^{\alpha,2}(C), 1\le i<j\le d\},\\
	&\vdots\\
	&T_d(\alpha;C)=\mathcal{G}_2^{\alpha,d}(C),
	\end{aligned}
\end{equation*}
where $\alpha\ge 1$ and $C>0$. Note that $T_1(\alpha;C)$ contains additive functions and $T_r(\alpha;C),1<r\le d$ have higher order interactions. The $L_2$ metric entropies of these classes are of the same orders as $\mathcal{G}_2^{\alpha,1}(C),\mathcal{G}_2^{\alpha,2}(C),\cdots,\mathcal{G}_2^{\alpha,d}(C)$, respectively, which are $\epsilon^{-\frac{1}{\alpha}},\epsilon^{-\frac{2}{\alpha}},\cdots,\epsilon^{-\frac{d}{\alpha}}$, respectively (see, e.g, \cite{Yang4}).

By \Cref{th1}, we have the following result.
\begin{corollary}
	For the nonparametric regression function over class $T_r(\alpha;C)$, the minimax rate of convergence under the squared $L_2$ loss for estimating the regression function is $n^{-\frac{2\alpha}{2\alpha+r}}$ for all $\alpha\ge 1$ and $1\le r\le d$.
\end{corollary}

From the result above, we see that the minimax rate of convergence is much slower when the interaction order is high. Also, for the additive model, the rate of convergence is just decided by the marginal smoothness, which can significantly increase the estimation accuracy compared with the models that contain high-order interactions.

Instead of Sobolev classes considered above, similar results can be readily stated for H\"{o}lder classes, in which the smoothness parameter $\alpha$ is no longer  restricted to be a positive integer, but can be any positive number. 

\subsection{Multiple Index Models}
Differently from the additive or low interaction order modeling for the regression function, we may consider multiple index models instead in the location-scale model setting. For recent research on multiple index models, see e.g, \cite{ChenD} and \cite{MaS}. We here focus on estimation of the whole regression function instead of the linear space parameters or the one dimensional functions. For $p\ge 1$, the regression function is of the form
\begin{equation*}
	\lambda_1(\pmb\beta_{(1)}^{T}\pmb x)+\lambda_2(\pmb\beta_{(2)}^{T}\pmb x)+\cdots +\lambda_p(\pmb\beta_{(p)}^{T}\pmb x),
\end{equation*}
where $\pmb x\in [0,1]^d$ and $\lambda_i,1\le i\le p$ are non-constant one dimensional functions (the details will be given later). Moreover, $\pmb\beta_{(i)}\in \mathbb{R}^d,1\le i\le p$ are distinct $d$-dimensional vectors satisfying $\lVert\pmb\beta_{(i)}\rVert_2=1,1\le i\le p$. When $p=1$, it corresponds to the single index model. Since we are interested in the convergence rate of minimax estimation of the whole regression function, we are not concerned with the identifiability issues and do not specify the identifiability conditions for the model we study.

Assume $\lambda_i,1\le i\le p$, are non-constant and belong to a H\"{o}lder class $\Lambda^{s,\gamma}(L)$ including all of the functions satisfying $\gamma$-H\"{o}lder continuous condition on the $s$-th derivative of $\lambda$:
\begin{equation*}
	\lvert \lambda^{(s)}(z_1)-\lambda^{(s)}(z_2)\rvert\le L\lvert z_1-z_2\rvert^\gamma,
\end{equation*}
where $z_1,z_2\in\mathbb{R}$, $L$ is a constant, $s$ is a nonnegative integer and $0<\gamma\le 1$, and assume for each $k$, $0\le k\le s$, the $k$-th derivative of $\lambda$ is uniformly upper bounded. Let $\alpha=s+\gamma$. We know that the metric entropy of $\Lambda^{s,\gamma}(L)$ is $M_d(\varepsilon;\Lambda)\asymp \varepsilon^{-\frac{1}{\alpha}}$ under the $L_2$ distance as shown in \cite{Lorentz}.

\begin{corollary}\label{cor4}
    If the regression function is assumed to belong to the multiple index model over the H\"{o}lder classes $\Lambda^{s,\gamma}(L)$, the minimax rate of convergence under the squared $L_2$ loss for estimating the regression function is $n^{-\frac{2\alpha}{2\alpha+1}}$ for all $\alpha> 0$.
\end{corollary}

See \Cref{appendix} for the proof of \Cref{cor4}.

Through this example, we see that for the multiple index model with fixed $p$, the minimax rate of convergence is just decided by the overall smoothness parameters $\alpha$. This can significantly improve the convergence speed compared with models with $d$-variate smooth function $ \lambda^*(x_1,\cdots,x_d)$ with overall smoothness $\alpha$, in which case the convergence rate is only $n^{-\frac{2\alpha}{2\alpha+d}}$ as seen in the previous subsection. This reflects the well-known curse of dimensionality.

\section{Discussions}\label{feasible_estimator}
In this section, we discuss on relation of our work to the literature, our assumption on size of the scale function class relative to that of the target regression function class, and feasible estimation to achieve the minimax rate of convergence.

\subsection{Minimax lower rates under non-normal error distributions}\label{non_normal_lower_rates}

As pointed out in the introduction, in the literature, for regression problems, estimators have been developed for the estimation of the $\psi$-tile (mean, quantile or expectile) and shown to archive a certain rate of convergence. For claiming the estimators are minimax-rate optimal, typically the argument that the established upper rate matches the known minimax lower rate under the Gaussian assumption on the regression error is used. Such a claim is certainly valid in the case of the error distribution assumed in a class that includes the normal distribution. 
It is helpful to emphasize here that the worst-case performance in the minimax evaluation includes that under the normal error. 
However, this type of results have two limitations. First, it leaves the problem unresolved if the lower bound can actually be improved in rate under a class of error distributions that do not contain the normal distribution. For instance, consider a light tail distribution $f(y)=c_1\exp{(-c_2y^4)}$, $-\infty <y< \infty$, where $c_1$ and $c_2$ are positive constants. Clearly, its tail is much lighter compared to the normal distribution. The lower bounding results in the literature, to our knowledge, does not answer the question if the minimax rate of convergence for expectile or quantile regression under this error distribution is faster than the normal distribution. \Cref{th1} in this paper is applicable, which shows that the rates of convergence is identical to the normal regression case. Second, nor is it known that a severely heavy tailed error distribution would damage the minimax rate of convergence. In particular, to the best of our knowledge, for example, it is unknown if the minimax-rate of mean/expectile regression (or quantile regression) under an asymmetric Laplace distribution (or Cauchy distribution for quantile regression) stays the same as under the normal error. 
Our results in this paper fill in the aforementioned gaps in the literature and offer a general understanding on optimal $\psi$-tile regression in the location-scale framework. 

Another interesting point is the following. While our results show that heavy tail of the error distribution does not affect the minimax rate of convergence for quantile estimation, it should not be interpreted as that the error distribution does not really matter for optimal-rate convergence. It actually does! For instance, suppose the analyst knows that the error distribution is symmetric about 0 and wants to estimate the median (and mean if it exists). If the error is Cauchy but the analyst mistakenly treats the error as normal and consequently applies a normal mean regression tool to estimate the median function. Then, the estimators (e.g., histogram/kernel regression) typically may not converge at all due to the simple fact that the sample mean of Cauchy random variables has the same Cauchy distribution and thus local averaging of the responses does not even produce a consistent estimator. In contrast, an estimator that properly targets the median instead of the non-existing mean (such as our estimator in the derivation of the minimax upper bound, albeit difficult to implement) achieves the minimax rate of convergence. Therefore, in presence of heavy-tailed random errors, unlike the light-tail situation, it is critically important to choose the right regression method for optimal rate of convergence.    

\subsection{Massiveness of the scale functions relative to that of the regression functions}\label{scale_function_discussion}

For our main results, \Cref{con3} is needed, which requires that $M_{h_0}(\varepsilon;\Sigma)=O(M_{h_0}(\varepsilon;\mathcal{U}))$ as $\varepsilon\to 0$. It means that the class of the scale functions is no larger than that of the regression functions in terms of the metric entropy order. In our view, the condition is sensible for $\psi$-tile regression for the following two reasons.

First, for mean, expectile or quantile regressions, when the errors are suspected to be heteroskedastic, a common approach is to model the variance function parametrically for feasibility and interpretability purposes, and it has been successfully applied in various applications (see, e.g., \cite{Davidian}). This approach is also natural because in $\psi$-tile regression, the regression function is of primary interest and the error standard deviation (or scale) function is a nuisance and simplicity in its modeling is preferred. With a parametric family for the scale functions,  $M_{h_0}(\varepsilon;\Sigma)=O(M_{h_0}(\varepsilon;\mathcal{U}))$ certainly holds.

Second, when one is interested in $\psi$-tile regression at multiple $\tau$ levels, the condition $M_{h_0}(\varepsilon;\Sigma)=O(M_{h_0}(\varepsilon;\mathcal{U}))$ automatically holds. 
Given model (\ref{eq:1}) with an error distribution with density $f$, the conditional quantile, expectile, momentile and $\psi$-tile at arbitrarily level $\tau\in (0, 1)$ are all of the form of $\tilde{\mu}(x)=\mu(x)+c\sigma(x)$ for some constant $c$ that depends on $\tau$, $f$ and $\psi$. Unless $c=0$, which can happen for at most a single $\tau$ value, the metric entropy of the associated class of the regression function $\tilde{\mu}(x)$ is of order $\max(M_{h_0}(\varepsilon; {\cal{U}}), M_{h_0}(\varepsilon; \Sigma))$ as $\varepsilon\rightarrow 0$. Therefore, as long as one is interested in at least two $\tau$ levels, the regression function classes  are automatically as large as the class of the scale functions. 

It should be pointed out that the case $M_{h_0}(\varepsilon;\Sigma)$ is much larger than $O(M_{h_0}(\varepsilon;\mathcal{U}))$ can still be of interest. As a reviewer pointed out, parametric regression with independent errors with standard deviation $\sigma (x)$ uniformly bounded enjoys the usual parametric rate of convergence for estimation of the regression function. In this case, clearly the class of scale functions is really large (in fact with infinite metric entropy for a small enough packing radius). It is intriguing that the tremendous complexity of the class of scale functions actually does not damage the rate of convergence for the estimation of the regression function. On the other hand, as soon as one intends to estimate another $\psi$-tile, even if it is of the same nature but at a slightly different $\tau$ level, the rate of convergence is completely different. For example, under model (\ref{eq:1}) with the strong assumption of normal error and $E\varepsilon =0$, suppose $\cal{U}$ is a parametric function class and $\Sigma$ consists of all functions $\underline{\sigma}\leq \sigma (x)\leq \overline{\sigma}$ for some positive constants $  \underline{\sigma}  < \overline{\sigma}$. Then for mean regression, the parametric rate $1/n$ is achieved by standard methods such as least squares. However, for expectile or quantile regression with $\tau \neq 1/2$, the regression function is $u(x)+c\sigma (x)$ for some constant $c\neq 0$. Clearly, the resulting regression function class is as complex as $\Sigma$ in terms of metric entropy order and consequently we know there cannot exist any convergent estimator of the regression function in the uniform sense no matter how $\tau$ is close to $1/2$. In this example, the estimand of the parametric mean function is an ``outlier'' in the grand scheme of expectile or quantile  (or more generally $\psi$-tile) regression that intends to offer a broader understanding than a single sliced view at a fixed and unusually lucky $\tau$. The above said, the study of minimax estimation at an ``outlier'' $\tau$ is certainly of interest on its own, which we do not pursue in this work. 

\subsection{Estimators with minimax optimality}\label{minimax_estimator}
In this paper, our focus is on developing a general theoretical understanding on unifying minimax mean, expectile, and quantile regressions in the location-scale nonparametric framework. As such, we highlight the most essential aspects in their generality on the determination of the minimax rate of convergence of regression estimation. In particular, as far as the nature of the regression function class is concerned, it turns out that its metric entropy order alone is enough and other characteristics of the function class are irrelevant as far as  the minimax rate of convergence is concerned. For establishing the minimax rate of convergence, we need to develop estimators of the regression function whose worst-case risks match the minimax lower bounds in order. To this end, given the generality of our results that do not require knowledge of the regression function class beyond the metric entropy order, the construction of the estimators is of a theoretical nature that is in tune with the abstract metric entropy characterization of the class $\cal{U}$. Therefore, not surprisingly, these 
theory-oriented $\epsilon$-net based estimators are not easy for real implementation. In a specific application, typically one works with a concrete function class (or classes) with known helpful characterizations of the regression functions (such as smoothness, monotonicity, additivity, etc.). Then while our theory provides a precise understanding of the minimax rate of convergence, the knowledge of the specific  function classes can be naturally taken into account in deriving practical estimators that achieve the minimax rate of convergence. An example is given below to illustrate this point.     

Consider expectile regression under the asymmetric Laplace error distribution, with the expectile function $f$ assumed to be bounded and belong to the Besov class $B_{2,\infty}^{\alpha}$ on $[-1,1]^d$, whose metric entropy order is $\epsilon^{-d/\alpha}$, where $\alpha >0$ is the smoothness parameter (see, e.g., \cite{DeVore}). Based on \Cref{th1} and \Cref{cor1} in this paper, we know the minimax rate of convergence for estimating the regression function is $n^{-\frac{2\alpha}{2\alpha+d}}$ (and the same rate also holds under the other error distributions studied in \Cref{error_qr}). As pointed out earlier, our $\epsilon$-net based estimators in deriving the minimax result is not practical in its implementation.

\cite{Farooq} proposed a support vector machine (SVM) method for estimation. Given a regularization parameter $\lambda>0$, a given $\tau$ in $(0,1)$ and a reproducing kernel Hilbert space $H$ on $[-1,1]^d$ with a bounded, measurable kernel $k: [-1,1]^d\times [-,1]^d\to \mathbb{R}$, the SVM method builds an estimate $u_\lambda$ by solving an optimization problem of the form
\begin{equation}\label{svm}
	u_\lambda=\arg \underset{u\in H}{\min}\big ( \lambda \lVert u\rVert_H^2 + \frac{1}{n}\sum_{i=1}^nL (y_i,u(x_i)\big),
\end{equation}
where $L$ is the ALS for ER.

They obtained an algorithm for solving \eqref{svm} using Gaussian radial basis kernels and established the learning rate $n^{-\frac{2\alpha}{2\alpha+d}}\log n$ of the above estimator for any error distribution decaying exponentially fast. Therefore, their practically feasible estimator achieves the minimax rate of convergence up to an extra logarithmic factor. It remains to be seen how to come up with a computationally feasible algorithm to achieve the minimax rate of convergence exactly, especially when $\alpha$ is unknown.

\section{Conclusion}\label{conclusion}

Location-scale models cover familiar mean, quantile and expectile regression problems that are widely used in econometric and statistical applications. While a general understanding on minimax rate for normal mean regression is available in the literature, prior to our work, little was given on the minimax optimal rates for estimating a conditional quantile or expectile function in a general function class under non-Gaussian errors.  In this paper, in the location-scale model framework, we have derived the minimax rates of convergence for regression learning under the square $L_2$ loss for nonparametric function classes. Not surprisingly, metric entropy continues to play a central role in the determination of the optimal rate of convergence. The results are readily applicable for various function classes that are commonly considered in data science even when the error distribution does not have a finite mean.

A limitation of our work is that it deals with i.i.d.\ observations. \cite{ChenXiaohong} showed that in some settings similar regression results hold under weakly dependent random errors. Under Gaussianity, effects of short- and long-range dependences of the errors are well understood for minimax nonparametric regression with homoscedastic errors (see, e.g., \cite{Yang2001}, \cite{YangYang2019}). It remains open to understand if a general error distribution with heteroscedasticity can be handled in the same spirit that short-range dependences do not damage the minimax rate of convergence for $\psi$-tile regression but long-range dependence of the errors, when strong enough relative to the massiveness of the regression function class, can substantially slow down the rate of convergence as shown in the aforementioned articles. Another limitation of our work is that unlike results in \cite{Horowitz2} and  \cite{XiaohongChen}, our framework does not deal with endogeneity in the regression relationship. It remains to be seen how to modify our  minimax upper and lower bounding techniques for identifying the minimax rate of convergence in a generality similar to the present work when instrumental variables are included.

\appendix
\numberwithin{equation}{section}
\section{Appendix}\label{appendix}
\subsection{A lemma on Kullback-Leibler divergence, Hellinger distance and $L_2$ distance} We need a lemma for the proof of \Cref{th1}.
For the following result, $\mu$ is taken to be a constant and $f_{\mu,\sigma} (y)=\frac{1}{\sigma}f(\frac{y-\mu}{\sigma})$.
    
\begin{lemma} \label{relating-to-L2} For the error distribution family, under \Cref{con1} and \Cref{con2} respectively, we have the following characterizations of the K-L divergence and Hellinger distance.
    \begin{enumerate}
    	\item Under \Cref{con2}, for any $J>0$, $\overline{\sigma}>1$ and $\underline{\sigma}=\overline{\sigma}^{-1}$, we can find a constant $\underline{C}>0$ (which may depend on $f$, $J$ and $\overline{\sigma}$) such that
    	\begin{equation*}
    		d^2_H(f_{\mu_1,\sigma_1},f_{\mu_{2},\sigma_2})\ge \underline{C}(\mu_1-\mu_{2})^2
    	\end{equation*}
    	holds for all $\mu_1,\mu_2\in [-J,J]$ and $\sigma_1,\sigma_2\in [\underline{\sigma},\overline{\sigma}]$.
        \item Under \Cref{con1}, for any $J>0$, $\overline{\sigma}>1$ and $\underline{\sigma}=\overline{\sigma}^{-1}$, there exists a constant $\overline{C}>0$ (which may depend on $f$, $J$ and $\overline{\sigma}$) such that
    	\begin{equation*}
    		D(f_{\mu_1,\sigma_1}||f_{\mu_{2},\sigma_2})\leq \overline{C}[(\mu_1-\mu_{2})^2+(\sigma_1-\sigma_2)^2]
    	\end{equation*}
    	holds for all $\mu_1,\mu_2\in [-J,J]$ and $\sigma_1,\sigma_2\in [\underline{\sigma},\overline{\sigma}]$.
    \end{enumerate}
\end{lemma}
\begin{proof}      

     We first prove part 1. Suppose the conclusion does not hold. Then, there exists a sequence $C_n\to 0$, and we can find $\mu_{1n},\mu_{2n}$ in $[-J,J]$ and $\sigma_{1n},\sigma_{2n}\in [\underline{\sigma},\overline{\sigma}]$ such that
     \begin{equation}
     \label{1}
    		d^2_H(f_{\mu_{1n},\sigma_{1n}},f_{\mu_{2n},\sigma_{2n}})< C_n(\mu_{1n}-\mu_{2n})^2.
     \end{equation}
     %

     Since Hellinger distance is location and scale invariant, we have
     \begin{equation*}
     	d^2_H(f_{{\mu_{1n}},\sigma_{1n}},f_{\mu_{2n},\sigma_{2n}})=d^2_H(f_{{0},1},f_{\kappa_n,\varpi_n}),
     \end{equation*}
     where $\kappa_n=\frac{\mu_{2n}-\mu_{1n}}{\sigma_{1n}}$ and $\varpi_n=\frac{\sigma_{2n}}{\sigma_{1n}}$.

     Since $\{(\kappa_n,\varpi_n), n\ge 1\}$ are bounded, we must have a convergent subsequence $n_k$ such that $\kappa_{n_k}\to \kappa$ and $\varpi_{n_k}\to \varpi$ for some $\kappa\in [-\frac{2J}{\underline{\sigma}},\frac{2J}{\underline{\sigma}}]$ and $\varpi\in [\underline{\sigma}/\overline{\sigma},\overline{\sigma}/\underline{\sigma}]$. If $\kappa_{n_k}\to \kappa \neq 0$ or $\varpi_{n_k}\to\varpi\neq 1$, since our location-scale family is continuous under the Hellinger distance, together with that our location-scale family is identifiable, we have
     \begin{equation*}
     	\underset{n\to\infty}{\lim}d^2_H(f_{0,1},f_{\kappa_{n_k},\varpi_{n_k}})=d^2_H(f_{0,1},f_{\kappa,\varpi})>0,
     \end{equation*}
     which contradicts \eqref{1} since $C_n\to 0$. Thus we must have $\kappa_{n_k}\to \kappa=0$ and $\varpi_{n_k}\to \varpi=1$.

    In the following, we denote $\sqrt{f_{\kappa,\varpi}}$ as $\xi_{\kappa,\varpi}$. Based on \Cref{con2}, taking Hellinger derivative of $\xi_{\kappa_{n_k},\varpi_{n_k}}$ at $\xi_{0,1}$ we have
    \begin{equation*}
	\xi_{\kappa_{n_k},\varpi_{n_k}}=\xi_{0,1}+(\kappa_{n_k},\varpi_{n_k}-1)\dot{\xi}_{0,1}+r_{\kappa_{n_k},\varpi_{n_k}},
    \end{equation*}
    where $\dot{\xi}_{0,1}=(\dot{\xi}_{(0,1)_\kappa},\dot{\xi}_{(0,1)_\varpi})^T$ is a two dimensional vector, $\frac{\left\|r_{\kappa_{n_k},\varpi_{n_k}}\right\|}{\sqrt{\kappa_{n_k}^2+(\varpi_{n_k}-1)^2}}=o(1)$, $\lVert\dot{\xi}_{(0,1)_\kappa}\rVert>0$ and $\lVert\dot{\xi}_{(0,1)_\varpi}\rVert>0$. Then we obtain
    \begin{equation*}
    \begin{aligned}
    d_H^2(f_{0,1},f_{\kappa_{n_k},\varpi_{n_k}})
    &=\int(\xi_{{0},1}-\xi_{\kappa_{n_k},\varpi_{n_k}})^2dy \\
    &=\int(\kappa_{n_k}\dot{\xi}_{(0,1)_\kappa}+(\varpi_{n_k}-1)\dot{\xi}_{(0,1)_\varpi}+r_{\kappa_{n_k},\varpi_{n_k}})^2dy \\
    &=\int(\kappa_{n_k}\dot{\xi}_{(0,1)_\kappa}+(\varpi_{n_k}-1)\dot{\xi}_{(0,1)_\varpi})^2dy+\int(r_{\kappa_{n_k},\varpi_{n_k}}(y))^2dy\\
    &~~~~+2\int(\kappa_{n_k}\dot{\xi}_{(0,1)_\kappa}+(\varpi_{n_k}-1)\dot{\xi}_{(0,1)_\varpi})r_{\kappa_{n_k},\varpi_{n_k}}dy.\\
    \end{aligned}
    \end{equation*} 
    Since $\dot{\xi}_{0,1_\kappa}(y)$ and $\dot{\xi}_{0,1_\varpi}(y)$ are not linear related, as $k \rightarrow \infty$, we get
    $$
    		\int(\kappa_{n_k}\dot{\xi}_{(0,1)_\kappa}+(\varpi_{n_k}-1)\dot{\xi}_{(0,1)_\varpi})^2dy \asymp \kappa_{n_k}^2+(\varpi_{n_k}-1)^2 \succeq \kappa_{n_k}^2.
$$

    By Cauchy-Schwarz inequality, we have 
    \begin{equation*}
    \begin{aligned}
    	    2\int(\kappa_{n_k}\dot{\xi}_{(0,1)_\kappa}+(\varpi_{n_k}-1)\dot{\xi}_{(0,1)_\varpi})r_{\kappa_{n_k},\varpi_{n_k}}dy&\le 2 \lVert \kappa_{n_k}\dot{\xi}_{(0,1)_\kappa}+(\varpi_{n_k}-1)\dot{\xi}_{(0,1)_\varpi}\rVert\lVert r_{\kappa_{n_k},\varpi_{n_k}}\rVert\\
    	    &=O(\sqrt{\kappa_{n_k}^2+(\varpi_{n_k}-1)^2})*o((\sqrt{\kappa_{n_k}^2+(\varpi_{n_k}-1)^2)})\\
    	    &=o(\kappa_{n_k}^2+(\varpi_{n_k}-1)^2).
    \end{aligned}
    \end{equation*}
    Therefore, $d_H^2(f_{0,1},f_{\kappa_{n_k},\varpi_{n_k}})$ is at least of order $\kappa_{n_k}^2$, which contradicts \cref{1}. This proves the first part of the lemma.
    
    Now, we prove the second part. Suppose the conclusion does not hold. Then, there exists a sequence $C_n\to \infty$, and we can find $\mu_{1n},\mu_{2n}$ in $[-J,J]$ and $\sigma_{1n},\sigma_{2n}\in [\underline{\sigma},\overline{\sigma}]$ such that
     \begin{equation*}
    		D(f_{\mu_{1n},\sigma_{1n}}||f_{\mu_{2n},\sigma_{2n}})> C_n[(\mu_{1n}-\mu_{2n})^2+(\sigma_{1n}-\sigma_{2n})^2].
     \end{equation*}

     We also have
     \begin{equation*}
     	D(f_{{\mu_{1n}},\sigma_{1n}}||f_{\mu_{2n},\sigma_{2n}})=D(f_{{0},1}||f_{\kappa_n,\varpi_n}),
     \end{equation*}
     where $\kappa_n=\frac{\mu_{2n}-\mu_{1n}}{\sigma_{1n}}$ and $\varpi_n=\frac{\sigma_{2n}}{\sigma_{1n}}$.

     Since $\{(\kappa_n,\varpi_n), n\ge 1\}$ are bounded, we must have a convergent subsequence $n_l$ such that
     \begin{equation*}
     	D(f_{{0},1}||f_{\kappa_{n_l},\varpi_{n_l}})>C_{n_l}[\kappa_{n_l}^2+(1-\varpi_{n_l}^2)],
     \end{equation*}
     which contradicts \Cref{con1}, based on the earlier arguments.
\end{proof}

\subsection{Proof of \Cref{th1}}
\begin{proof}
	To obtain the minimax rate of convergence, we shall derive lower and upper bounds that match in order under the squared $L_2$ distance.

	Let $V_K(\varepsilon)$ denote the covering $\varepsilon$-entropy of $\{p_{\mu,\sigma,h}:\mu\in \mathcal{U},\sigma\in \Sigma\}$ under the square root of K-L divergence (recall that the covering $\varepsilon$-entropy is the logarithm of the minimum cardinality of $\varepsilon$-nets), where $p_{\mu,\sigma,h}$ is the joint density of $(X,Y)$ with marginal density $h(x)$ for $X$, regression function $\mu(x)$ and scale function $\sigma(x)$. Let $M_{h_0}(\varepsilon;\mathcal{U})$ be the $\varepsilon$-packing entropy of $\mathcal{U}$ under the $L_2(h_0)$ distance. Based on part 2 of \Cref{relating-to-L2}, under the assumptions that the regression functions in the class $\mathcal{U}$ are uniformly bounded and the scale functions in the class $\Sigma$ are uniformly upper and lower bounded, we know 
	\begin{equation}\label{kl_upper}
		\begin{aligned}
			D(p_{\mu_1,\sigma_1,h}|| p_{\mu_2,\sigma_2,h})&=D(h(x)f_{\mu_1,\sigma_1}(y)|| h(x)f_{\mu_2,\sigma_2}(y))\\
			&\le c(\left\|{\mu_1-\mu_2}\right\|^2_{2,h}+\left\|{\sigma_1-\sigma_2}\right\|^2_{2,h})\\
			&\le cC(\left\|{\mu_1-\mu_2}\right\|^2_{2,h_0}+\left\|{\sigma_1-\sigma_2}\right\|^2_{2,h_0}),
			\end{aligned}
	\end{equation}
	where the constant $c$ depends on the density $f$, constant $J=\underset{\mu\in\mathcal{U}}{\sup}\lVert\mu\rVert_\infty$ and $\overline{\sigma}$ (recall $1/\overline{\sigma}\le \sigma(x)\le \overline{\sigma}$ for all $\sigma\in \Sigma$), and $C$ is the constant in \Cref{con4}.
	
	Therefore, we know $V_K(\varepsilon)$ is upper bounded by order of $M_{h_0}(\tilde{c}\varepsilon;\mathcal{U})+M_{h_0}(\tilde{c}\varepsilon;\Sigma)$ for some constant $\tilde{c}>0$, which has the same order as $M_{h_0}(\varepsilon;\mathcal{U})$ by \Cref{con3}. In addition, we can construct a covering set for $\{\mathcal{U},\Sigma\}$ under $h_0$. Let $\varepsilon_n$ and $\underline\varepsilon_{n}$ be determined by
	\begin{equation*}
		V_K(\varepsilon_n)=n\varepsilon_n^2,
	\end{equation*}
	\begin{equation*}
		M_{h_0}(\underline \varepsilon_{n};\mathcal{U})=4n\varepsilon_n^2+2\log 2.
	\end{equation*}

	Below we first give a risk upper bound of relevant estimation in terms of the Hellinger distance and then show the squared $L_2(h)$ risk on estimation of $\mu$ is actually upper bounded by a multiple of the Hellinger risk.

    We connect the joint density estimation with regression. Suppose we have a density estimator $\hat p_n(x,y)$ of the joint density of $(\pmb X, \pmb Y)$, which is of the form $h(x)\hat q_n(y|x)$. Let $(\hat \mu,\hat\sigma)$ be obtained by minimizing $d_H\big(\hat q_n(y|x),\frac{1}{\tilde\sigma(x)}f(\frac{y-\tilde\mu(x)}{\tilde\sigma(x)})\big)$ over $\underline{\sigma}\le \tilde\sigma(x)\le \overline{\sigma}$ and $-J\le \tilde\mu(x)\le J$ at each $x\in \cal X$. Then because $d_H(p_{\mu,\sigma,h},\hat p_n)\ge d_H(p_{\hat\mu,\hat\sigma, h},\hat p_n)$, we have
$$d_H(p_{\mu,\sigma,h},p_{\hat\mu,\hat\sigma, h})\le d_H(p_{\mu,\sigma,h},\hat p_n)+d_H(p_{\hat\mu,\hat\sigma,h},\hat p_n)
    \le 2d_H(p_{\mu,\sigma,h},\hat p_n).$$  
    
    We can construct an estimator $\hat p_n$ satisfying $\sup_{\mu\in\mathcal{U},\sigma\in\Sigma,h\in\mathcal{H}}\mathbb{E}D(p_{\mu,\sigma,h}||\hat{p}_n)\leq 2\varepsilon_n^2$. To be more specific on how to construct the estimator $\hat{p}_n$, let $A_n$ and $B_n$ be $(c^*\varepsilon_n)$-nets of $\mathcal{U}$ and $\Sigma$, respectively, under $\lVert\cdot\rVert_{2,h_0}$, where $c^*$ is a constant such that for any $(\mu,\sigma)\in \mathcal{U}\times \Sigma$, there exists $\tilde{\mu}\in A_n$ and $\tilde\sigma \in B_n$ s.t. $D(p_{\mu,\sigma,h}||p_{\tilde{\mu},\tilde{\sigma},h})\le \varepsilon^2_n$ based on \eqref{kl_upper}. Then, $G_{\varepsilon_n}=A_n\times B_n$ is an $\varepsilon_n$-net of $\{\mathcal{U},\Sigma\}$ under the square root of K-L divergence. For each pair of $(\mu,\sigma)\in \mathcal{U}\times \Sigma$, there exists $p_j:=p_{\mu_j,\sigma_j,h}$ with $(\mu_j,\sigma_j)\in G_{\varepsilon_n}$ such that $D(p_{\mu,\sigma,h}||p_j)\le \varepsilon_n^2$. Let $ p^n:=p(x_1,y_1,x_2,y_2,\cdots,x_n,y_n)=\frac{1}{\lvert G_{\varepsilon_n}\rvert}\sum_{i=1}^{\lvert G_{\varepsilon_n}\rvert} p_i(x_1,y_1)p_i(x_2,y_2)\cdots p_i(x_n,y_n)$, we have 
    \begin{equation*}
    	\begin{aligned}
    		&D(p^n_{\mu,\sigma,h}||p^n)&\\
    		&=\int p_{\mu,\sigma,h}(x_1,y_1)\cdots p_{\mu,\sigma,h}(x_n,y_n)\log \frac{p_{\mu,\sigma,h}(x_1,y_1)\cdots p_{\mu,\sigma,h}(x_n,y_n)}{\frac{1}{{\lvert G_{\varepsilon_n}\rvert}}\sum_{i=1}^{\lvert G_{\varepsilon_n}\rvert} p_i(x_1,y_1)\cdots p_i(x_n,y_n)}dx_1dy_1\cdots dx_ndy_n\\
    		&\le \int p_{\mu,\sigma,h}(x_1,y_1)\cdots p_{\mu,\sigma,h}(x_n,y_n)\log \frac{p_{\mu,\sigma,h}(x_1,y_1)\cdots p_{\mu,\sigma,h}(x_n,y_n)}{\frac{1}{{\lvert G_{\varepsilon_n}\rvert}} p_j(x_1,y_1)\cdots p_j(x_n,y_n)}dx_1dy_1\cdots dx_ndy_n\\
    		&=\log {\lvert G_{\varepsilon_n}\rvert}\\
    		&~~~~+\int p_{\mu,\sigma,h}(x_1,y_1)\cdots p_{\mu,\sigma,h}(x_n,y_n)\log \frac{p_{\mu,\sigma,h}(x_1,y_1)\cdots p_{\mu,\sigma,h}(x_n,y_n)}{p_j(x_1,y_1)\cdots p_j(x_n,y_n)}dx_1dy_1\cdots dx_ndy_n\\
    		&=\log {\lvert G_{\varepsilon_n}\rvert}+nD(p_{\mu,\sigma,h}||p_j)\\
    		&\le V_K(\varepsilon_n)+n\varepsilon_n^2.
    	\end{aligned}
    \end{equation*}

    Note $p(x_1,y_1,\cdots,x_n,y_n)=\tilde{p}_1(x_1,y_1)\tilde{p}_2(x_2,y_2|x_1,y_1)\cdots \tilde{p}_n(x_n,y_n|x^{n-1},y^{n-1})$, where $x^{k-1}=(x_1,\cdots,x_{k-1})$ and $y^{k-1}=(y_1,\cdots,y_{k-1})$ for $2\le k\le n$. Let $\hat{p}_n=\frac{1}{n}[\tilde{p}_1(x,y)+\tilde{p}_2(x,y|x_1,y_1)+\cdots +\tilde{p}_n(x,y|x^{n-1},y^{n-1})]$, we have
    \begin{equation*}
    	\begin{aligned}
    		\mathbb{E}D(p_{\mu,\sigma,h}||\hat{p}_n)&=\mathbb{E}\int p_{\mu,\sigma,h}\log\frac{p_{\mu,\sigma,h}}{\hat{p}_n}dxdy\\
    		&\le\mathbb{E}\frac{1}{n}\sum_{i=1}^n\int p_{\mu,\sigma,h}(x_i,y_i)\log\frac{p_{\mu,\sigma,h}(x_i,y_i)}{\tilde{p}_i(x_i,y_i|x^{i-1},y^{i-1})}dx_idy_i\\
    		&=\frac{1}{n}\sum_{i=1}^n\mathbb{E}\log\frac{p_{\mu,\sigma,h}(x_i,y_i)}{\tilde{p}_i(x_i,y_i|x^{i-1},y^{i-1})}\\
    		&=\frac{1}{n}\mathbb{E}\log\frac{p_{\mu,\sigma,h}(x_1,y_1)\cdots p_{\mu,\sigma,h}(x_n,y_n)}{\tilde{p}_1(x_1,y_1)\tilde{p}_2(x_2,y_2|x_1,y_1)\cdots \tilde{p}_n(x_n,y_n|x^{n-1},y^{n-1}))}\\
    		&=\frac{1}{n}D(p^n_{\mu,\sigma,h}||p^n)\\
    		&\le \frac{V_K(\varepsilon_n)}{n}+\varepsilon_n^2\\
    		&=2\varepsilon_n^2.
    	\end{aligned}
    \end{equation*}
    
    It is important to note that $\hat p_n$ is of the form $h(x)\hat q_n(y|x)$ and $h$ needs not to be known in the construction of $\hat q_n(y|x)$.
    
    Since the Hellinger distance is upper bounded by the square root of K-L divergence, we have $\sup_{\mu\in\mathcal{U},\sigma\in \Sigma,h\in\mathcal{H}}\mathbb{E}d_H^2(p_{\mu,\sigma,h},\hat{p}_n)\leq 2\varepsilon_n^2$. It follows that $\sup_{\mu\in\mathcal{U},\sigma\in\Sigma,h\in\mathcal{H}}\mathbb{E}d_H^2(p_{\mu,\sigma,h},p_{\hat{\mu},\hat\sigma,h}) \le 4\sup_{\mu\in\mathcal{U},\sigma\in\Sigma,h\in\mathcal{H}}\mathbb{E}d_H^2(p_{\mu,\sigma,h},\hat{p}_n)\leq 8\varepsilon_n^2$.

    Now we show that the squared $L_2$ distance is upper bounded by a multiple of the Hellinger distance through Hellinger differentiability. By part 1 of \Cref{relating-to-L2} we have for some $\underline{C}>0$,
    \begin{equation*}
    	d_H^2(f_{\mu (x),\sigma(x)},f_{\hat{\mu} (x),\hat\sigma(x)})\ge\underline{C}(\mu(x)-\hat{\mu}_n(x))^2.
    \end{equation*}
    
    \noindent
    Furthermore, we have
    \begin{equation*}
    	\begin{aligned}
    		d_H^2(p_{\mu,\sigma,h},\hat p_n)&\ge \frac{1}{4}d_H^2(p_{\mu,\sigma,h},p_{\hat{\mu},\hat\sigma, h})\\
    		&=\frac{1}{4}\int d_H^2(f_{\mu (x),\sigma(x)},f_{\hat{\mu} (x),\hat\sigma(x)})h(x)dx\\
    		&\ge \frac{1}{4}\underline{C}\int (\mu(x)-\hat{\mu}_n(x))^2h(x)dx\\
    		&=\frac{1}{4}\underline{C}\lVert\mu-\hat\mu\rVert_{2,h}^2.
    	\end{aligned}
    \end{equation*}
    
    From all above, we conclude
    	$\mathbb{E}\lVert\mu-\hat\mu\rVert_{2,h}^2\preceq \varepsilon_n^2$.

	Then, let's derive the lower bound. It suffices to assume $\sigma$ is a fixed constant assumed to be in $\Sigma$ and focus on $h_0$ only, because the lower bound already matches the upper bound in order. With more choices in $\sigma$ and $h$, the lower bound cannot be smaller. Let $N_{\underline{\varepsilon}_n}$ be the $\underline{\varepsilon}_n$-packing set of $\mathcal{U}$ under the $L_2(h_0)$ distance.  By Fano's inequality (See Theorem 1 in \cite{Yang}), we have
	\begin{equation*}
		\underset{\hat\mu}{\inf}~ \underset{\mu\in \mathcal{U}}{\sup}\mathbb{E}[\lVert\mu-\hat\mu\rVert_{2,h_0}^2] \ge \frac{1}{4}(1-\frac{I(N_{\underline{\varepsilon}_n},Y^n)+\log 2}{\log \lvert N_{\underline{\varepsilon}_n}\rvert})\underline{\varepsilon}_n^2,
	\end{equation*}
	where $I(N_{\underline{\varepsilon}_n},Y^n)$ is the Shannon's mutual information between the uniformly distributed random parameter on $N_{\underline{\varepsilon}_n}$ and the random sample. We have
	\begin{equation*}
		\begin{aligned}
			I(N_{\underline{\varepsilon}_n},Y^n)=\sum_{\mu\in N_{\underline{\varepsilon}_n}} \omega(\mu)\int p(y^n|\mu)\log\frac{p(y^n|\mu)}{p_\omega(y^n)}dy^n,
		\end{aligned}
	\end{equation*}
	in which $\omega(\mu)$ is the uniform distribution on $\mu\in N_{\underline{\varepsilon}_n}$ and $p_\omega(y^n)=\sum_{\mu\in N_{\underline{\varepsilon}_n}} \omega(\mu)p(y^n|\mu)$ is the marginal distribution of $y$. Since the Bayes mixture density $p_\omega(y^n)$ minimize $I(N_{\underline{\varepsilon}_n},Y^n)$ over other candidates of marginal distribution $q(y^n)$, we have

		$$	I(N_{\underline{\varepsilon}_n},Y^n)\le \sum_{\mu\in N_{\underline{\varepsilon}_n}} \omega(\mu)\int p(y^n|\mu)\log\frac{p(y^n|\mu)}{q(y^n)}dy^n
			\le \max_{\mu\in N_{\underline{\varepsilon}_n}} D(p_\mu^n||q^n),$$
	where we denote $p_{\mu,\sigma,h_0}$ as $p_\mu$, because $\sigma$ and $h_0$ are fixed.

	Let $\tilde\omega$ be the uniform distribution on $G_{\varepsilon_n}$ and let $q(y^n)=p_{\tilde{\omega}}(y^n)=\sum_{\mu\in N_{\underline{\varepsilon}_n}} \tilde\omega(\mu)p(y^n|\mu)$, we have that for any $\mu\in \mathcal{U}$,
	\begin{equation*}
		\begin{aligned}
			D(p_\mu^n||q^n)&= E\log\frac{p(y^n|\mu)}{\frac{1}{\lvert G_{\varepsilon_n}\rvert}\sum_{\mu^\prime\in G_{\varepsilon_n}}p(y^n|\mu^\prime)}\\
			&\le \log \lvert G_{\varepsilon_n}\rvert + E\log\frac{p(y^n|\mu)}{p(y^n|\mu^\prime)}\\
			&= \log \lvert G_{\varepsilon_n}\rvert + D(p_\mu^n||p_{\mu^\prime}^n)\\
			&\le V_K(\varepsilon_n)+n\varepsilon_n^2.
		\end{aligned}
	\end{equation*}

	Hence, we have
	\begin{equation*}
		1-\frac{I(N_{\underline{\varepsilon}_n},Y^n)+\log 2}{\log \lvert N_{\underline{\varepsilon}_n}\rvert}\ge \frac{1}{2},
	\end{equation*}
	which leads to
	\begin{equation*}
		\underset{\hat\mu}{\inf}~ \underset{\mu\in \mathcal{U}}{\sup}\mathbb{E}[\lVert\mu-\hat\mu\rVert_{2,h_0}^2] \ge \frac{1}{8}\underline{\varepsilon}_n^2.
	\end{equation*}

	It follows
	\begin{equation*}
		\underset{\hat\mu}{\inf}~ \underset{\mu\in \mathcal{U},\sigma\in \Sigma, h\in \mathcal{H}}{\sup}\mathbb{E}[\lVert\mu-\hat\mu\rVert_{2,h}^2] \ge \frac{1}{8}\underline \varepsilon_{n}^2.
	\end{equation*}
	
	Under Conditions 1, 2 and 4, we know $V_K(\varepsilon_n)$ and $M_{h_0}(\varepsilon_n; \mathcal{U})$ are actually of the same order and the resulting $\varepsilon_n$ and $\underline{\varepsilon}_n$ converge at the same rate. This completes the proof of \Cref{th1}. 
\end{proof}

\subsection{Proof of Corollary 4}
\begin{proof}
    We define the parametric class
    \begin{equation*}
	\mathcal{B}=\{\pmb\beta^T\pmb x |\pmb x\in [0,1]^d, \pmb\beta\in \mathbb{R}^d, \lVert \pmb\beta\rVert_2=1\}.
    \end{equation*}

    Let $\mathcal{S}$ be an $\frac{\varepsilon}{2p}$-net for the H\"{o}lder class $\Lambda^{\alpha,\gamma}(L)$ under the $L_\infty$ distance with respect to the Lebesgue measure. Also, let $\mathcal{T}$ be an $(\frac{\varepsilon}{2pL_*})^{\frac{1}{\gamma}}$-net for the parametric class $\mathcal{B}$ under the $L_2$ distance with respect to the probability measure on $\pmb x$, in which $L_*$ is a constant and will be given later.
   
    Consider the class ${\cal{Q}}=\{\lambda_1(\pmb\beta_{(1)}^{T}\pmb x)+\lambda_2(\pmb\beta_{(2)}^{T}\pmb x)+\cdots +\lambda_p(\pmb\beta_{(p)}^{T}\pmb x): \lambda_i\in \Lambda^{\alpha,\gamma}(L), \pmb\beta_{(i)}\in \mathcal{B}, 1\le i\le p\}$, for the pairs $(\tilde{\lambda}_i,\pmb{\tilde{\beta}}_{(i)})$, in which $\tilde{\lambda}_i\in \mathcal{S}$ and $\pmb{\tilde{\beta}}_{(i)}\in \mathcal{T}$ with $1\le i\le p$, they form an $\varepsilon$-net for the class, as shown below.
    
    For $1\le i\le p$, for any $\lambda_i\in \Lambda^{\alpha,\gamma}(L)$, there exists $\tilde{\lambda}_i\in \mathcal{S}$ so that 
    \begin{equation*}
    	\begin{aligned}
    		\int_{[0,1]^d}\lvert \lambda_i(\pmb\beta_{(i)}^{T}\pmb x)-\tilde \lambda_i(\pmb\beta_{(i)}^{T}\pmb x)\rvert^2q_x(\pmb x)d\pmb x
    		&\le \frac{\varepsilon^2}{4p^2},
    	\end{aligned}
    \end{equation*}
    where $q_x(\pmb x)$ is the joint density of $\pmb x$.
    Since the $s$-th derivative of $\tilde \lambda_i$ satisfy $\gamma$-H\"{o}lder condition, we can easily verify that
    \begin{equation*}
    	\lvert \tilde \lambda_i(\pmb\beta_{(i)}^{T}\pmb x)-\tilde \lambda_i(\tilde{\pmb\beta}_{(i)}^{T}\pmb x)\lvert \le L_*\lvert \pmb\beta_{(i)}^{T}\pmb x-\tilde{\pmb\beta}_{(i)}^{T}\pmb x\rvert^\gamma,
    \end{equation*}
    where $L_*$ is a constant depending on $L$, $s$ and $d$. For $1\le i\le p$, since $\pmb\beta_{(i)}\in \mathcal{B}$ and $\pmb{\tilde{\beta}}_{(i)}\in \mathcal{T}$, we obtain
    \begin{equation*}
    	\begin{aligned}
    		\int_{[0,1]^d}\lvert\tilde \lambda_i(\pmb\beta_{(i)}^{T}\pmb x)-\tilde{\lambda}_i(\tilde{\pmb\beta}_{(i)}^{T}\pmb x)\rvert^2q_x(\pmb x)d\pmb x&\le L_*^2\int_{[0,1]^d}\lvert\pmb\beta_{(i)}^{T}\pmb x-\tilde{\pmb\beta}_{(i)}^{T}\pmb x\rvert^{2\gamma}q_x(\pmb x)\lambda(d\pmb x) \\
    		&\le L_*^2(\int_{[0,1]^d}\lvert\pmb\beta_{(i)}^{T}\pmb x-\tilde{\pmb\beta}_{(i)}^{T}\pmb x\rvert^{2}q_x(\pmb x)\lambda(d\pmb x))^\gamma\\
    		&\le \frac{\varepsilon^2}{4p^2}.
    	\end{aligned}
    \end{equation*}

    Hence, for $1\le i\le p$, we have
    \begin{equation*}
    \begin{aligned}
    	   &\int_{[0,1]^d}\lvert \lambda_i(\pmb\beta_{(i)}^{T}\pmb x)-\tilde{\lambda}_i(\tilde{\pmb\beta}_{(i)}^{T}\pmb x)\rvert^2q_x(\pmb x)d\pmb x\\
    	   &\le 2\int_{[0,1]^d}\lvert \lambda_i(\pmb\beta_{(i)}^{T}\pmb x)-\tilde \lambda_i(\pmb\beta_{(i)}^{T}\pmb x)\rvert^2q_x(\pmb x)\lambda(d\pmb x) +2\int_{[0,1]^d}\lvert\tilde \lambda_i(\pmb\beta_{(i)}^{T}\pmb x)-\tilde{\lambda}_i(\tilde{\pmb\beta}_{(i)}^{T}\pmb x)\rvert^2q_x(\pmb x)d\pmb x \\
    	   &\le \frac{\varepsilon^2}{p^2}.
    \end{aligned}
    \end{equation*}

    Therefore, we get
    \begin{equation*}
    \begin{aligned}
    	  &\int_{[0,1]^d}\lvert \lambda_1(\pmb\beta_{(1)}^{T}\pmb x)+\lambda_2(\pmb\beta_{(2)}^{T}\pmb x)+\cdots+ \lambda_p(\pmb\beta_{(p)}^{T}\pmb x)-\tilde{\lambda}_1(\tilde{\pmb\beta}_{(1)}^{T}\pmb x)-\tilde \lambda_2(\tilde{\pmb\beta}_{(2)}^{T}\pmb x)-\cdots -\tilde \lambda_p(\tilde{\pmb\beta}_{(p)}^{T}\pmb x)\lvert^2q_x(\pmb x)d\pmb x\\
    	  &\le p\int_{[0,1]^d}\lvert \lambda_1(\pmb\beta_{(1)}^{T}\pmb x)-\tilde{\lambda}_1(\tilde{\pmb\beta}_{(1)}^{T}\pmb x)\rvert^2q_x(\pmb x)d\pmb x+p\int_{[0,1]^d}\lvert \lambda_2(\pmb\beta_{(2)}^{T}\pmb x)-\tilde{\lambda}_2(\tilde{\pmb\beta}_{(2)}^{T}\pmb x)\rvert^2q_x(\pmb x)d\pmb x\\
    	  &~~~~+\cdots+p\int_{[0,1]^d}\lvert \lambda_p(\pmb\beta_{(p)}^{T}\pmb x)-\tilde{\lambda}_p(\tilde{\pmb\beta}_{(p)}^{T}\pmb x)\rvert^2q_x(\pmb x)d\pmb x\\
    	  &\le \varepsilon^2.
    \end{aligned}
    \end{equation*}

    Based on the known metric entropy order of the H\"{o}lder  class, we know the size of $\mathcal{S}$ is of order $(\frac{\varepsilon}{2p})^{-\frac{1}{\alpha+\gamma}}\asymp \varepsilon^{-\frac{1}{\alpha+\gamma}}$ and the size of $\mathcal{T}$ is of order $\log (\frac{2pL_*}{\varepsilon})^{\frac{1}{\gamma}}\asymp\log\frac{1}{\varepsilon}$. It follows
    that  the product of $\mathcal{S}$ and $\mathcal{T}$ is of size of order $\varepsilon^{-\frac{1}{\alpha+\gamma}}$. Therefore the metric entropy of $\cal{Q}$ is of order $\varepsilon^{-\frac{1}{\alpha+\gamma}}$.

    By \Cref{th1}, we have the minimax rate of convergence under the squared $L_2$ loss for estimating the regression function is $n^{-\frac{2(\alpha+\gamma)}{2(\alpha+\gamma)+1}}$ for all $\alpha\ge 0$ and $0<\gamma\le 1$.
\end{proof}

\subsection{A useful proposition on Hellinger differentiability and its extension}

The following result is helpful for showing Hellinger differentiability.

\begin{proposition}\label{prop1}
    (\cite{Pollard2}) Let $\mathcal L=\{l_\beta(x):\lvert\beta\rvert<J\}$ be a subset of nonnegative functions that are integrable with respect to Lebesgue measure $\lambda$ for some $J>0$.
    Suppose that \\
	(i) the map $(x,\beta)\mapsto l_\beta(x)$ is product measurable;\\
	(ii) the function $\beta\mapsto l_\beta(x)$ is absolutely continuous on $[-J,J]$ for almost all x;\\
	(iii) the function $\beta\mapsto l_\beta(x)$ has almost sure derivative $\dot{l}_\beta(x)$;\\
	(iv) for each $\beta$ the function $\dot{\xi}_\beta(x):=\frac{1}{2}\dot{l}_\beta(x)I(l_\beta(x)>0)/\sqrt{l_\beta(x)}$ is square integrable and $\mathbb{E} \dot{\xi}_\beta^2 \mapsto \mathbb{E} \dot{\xi}_0^2$ as $\beta \mapsto 0$.\\
	Then $\mathcal{L}$ has Hellinger derivative $\dot{\xi}(x)$ with respect to $\beta$ at 0.
\end{proposition}

We need the following lemma, which is an extension of \Cref{prop1}, to prove Corollaries 1 and 2.
\begin{lemma}\label{lm2}
	Let $g(y)$ denote the function from the location-scale family \eqref{loc_sca} with location parameter $\mu=0$ and scale parameter $\sigma=1$. Suppose that\\
	(i) the function $g(y)$ is absolutely continuous for almost all $y$;\\
	(ii) the function $g(y)$ has almost sure derivative $\dot g(y)$;\\
	(iii)  the function $\dot\xi_{(\mu,1)_\mu}(y):=-\frac{1}{2}\dot g(y-\mu)I(g(y-\mu)>0)/\sqrt{g(y-\mu)}$ is square integrable and $\mathbb{E}\dot \xi_{(\mu,1)_\mu}^2(y)\to \mathbb{E}\dot\xi_{(0,1)_\mu}^2(y)$ as $\mu\to 0$;\\
	(iv) $y\dot\xi_{(0,1)_\mu}(y)$ is square integrable;\\
	(v) $\lVert \dot\xi_{(0,\sigma)_\mu}(y)-\dot\xi_{(0,1)_\mu}(y)\lVert=O(\lvert \sigma-1\lvert)$ for $\sigma$ near 1.\\
    Then $\mathcal{F}$ has Hellinger derivative $(\dot\xi_{(0,1)_\mu}(y),\dot\xi_{(0,1)_\sigma}(y))^T$ with respect to $(\mu,\sigma)$ at $(0,1)$, where $\dot\xi_{(0,1)_\sigma}(y)$ is defined as $-\frac{1}{2}\sqrt{g(y)}-\frac{y}{2}\dot g(y)I(g(y)>0)/\sqrt{g(y)}$.
\end{lemma}
\begin{proof}
For the regular differentiability, we have
\begin{equation*}
	\begin{aligned}
		\frac{\partial f_{\mu,\sigma}}{\partial \mu}&=-\frac{1}{\sigma^2}\dot g(\frac{y-\mu}{\sigma}),\\
		\frac{\partial f_{\mu,\sigma}}{\partial \sigma}&=-\frac{1}{\sigma^2}g(\frac{y-\mu}{\sigma})-\frac{y-\mu}{\sigma^3}\dot g(\frac{y-\mu}{\sigma}).
	\end{aligned}
\end{equation*}

If the regular differentiation formulas were valid, the Hellinger derivative of $f_{\mu,\sigma}(y)$ with respect to $\mu$ would be 
\begin{equation*}
	\dot\xi_{(\mu,\sigma)_{\mu}}=-\frac{1}{2\sigma^{\frac{3}{2}}}\frac{\dot g(\frac{y-\mu}{\sigma})}{\sqrt{g(\frac{y-\mu}{\sigma})}}I(g(\frac{y-\mu}{\sigma})>0),
\end{equation*}
and the Hellinger derivative of $f_{\mu,\sigma}(y)$ with respect to $\sigma$ would be
\begin{equation*}
	\dot\xi_{(\mu,\sigma)_{\sigma}}=-\frac{1}{2\sigma^{\frac{3}{2}}}\sqrt{g(\frac{y-\mu}{\sigma})}-\frac{y-\mu}{2\sigma^{\frac{5}{2}}}\frac{\dot g(\frac{y-\mu}{\sigma})}{\sqrt{g(\frac{y-\mu}{\sigma})}}I(g(\frac{y-\mu}{\sigma})>0).
\end{equation*}

Hence, if $f_{\mu,\sigma}(y)$ is Hellinger differentiable with respect to $(\mu,\sigma)$ at $(0,1)$, we have \big(recall $\xi_{\mu,\sigma}=\frac{1}{\sqrt{\sigma}}\sqrt{g(\frac{y-\mu}{\sigma}})$\big)
\begin{equation*}
\begin{aligned}
	\xi_{\mu,\sigma}&=\xi_{0,1}+(\mu,\sigma-1)(\dot{\xi}_{(0,1)_\mu},\dot{\xi}_{(0,1)_\sigma})^T+r_{\mu,\sigma}.
\end{aligned}
\end{equation*}

Next, we verify that the above equations indeed are the Hellinger derivatives. We first verify that the Hellinger derivatives are square integrable. Given $(i)$-$(iii)$, by \Cref{prop1}, we have $f_{\mu,\sigma}(y)$ is Hellinger differentiable with respect to $\mu$ at 0, and we know that
\begin{equation*}
	\begin{aligned}
		\int \dot{\xi}^2_{(0,1)_\mu}dy<\infty.
	\end{aligned}
\end{equation*}
Also, we have
\begin{equation}\label{sigma_integrable}
	\begin{aligned}
	    &\int \dot{\xi}^2_{(0,1)_\sigma}dy\\
		&=\int (-\frac{1}{2}\sqrt{g(y)}+y\dot{\xi}_{(0,1)_\mu})^2dy\\
		&=\int [\frac{1}{4}g(y)+y^2\dot{\xi}^2_{(0,1)_\mu}-y\sqrt{g(y)}\dot{\xi}_{(0,1)_\mu}]dy.
	\end{aligned}
\end{equation}

By (iv) and Cauchy-Schwarz inequality, the above integral is finite. Hence, we have obtained that $\dot{\xi}_{(0,1)_\mu}$ and $\dot{\xi}_{(0,1)_\sigma}$ are square integrable.

Next, we verify that the residual term satisfies $r_{\mu,\sigma}=o(\sqrt{\mu^2+(\sigma-1)^2})$ near $(0,1)$. We have
\begin{equation*}
	\begin{aligned}
		\lVert r_{\mu,\sigma}(y)\rVert&=\lVert\xi_{\mu,\sigma}(y)-\xi_{0,1}(y)-\mu\dot\xi_{(0,1)_\mu}(y)-(\sigma-1)\dot\xi_{(0,1)_\sigma}(y)\rVert\\
		&=\lVert\xi_{\mu,1}-\xi_{0,1}-\mu\dot\xi_{(0,1)_\mu}+\xi_{0,\sigma}-\xi_{0,1}-(\sigma-1)\dot\xi_{(0,1)_\sigma}+\xi_{\mu,\sigma}-\xi_{0,\sigma}-\xi_{\mu,1}+\xi_{0,1}\rVert\\
		&\le \lVert\xi_{\mu,1}-\xi_{0,1}-\mu\dot\xi_{(0,1)_\mu}\rVert+\lVert\xi_{0,\sigma}-\xi_{0,1}-(\sigma-1)\dot\xi_{(0,1)_\sigma}\rVert+\lVert\xi_{\mu,\sigma}-\xi_{0,\sigma}-\xi_{\mu,1}+\xi_{0,1}\rVert,
	\end{aligned}
\end{equation*}
where the upper bound follows from Minkowski inequality. Since $f_{\mu,\sigma}(y)$ is Hellinger differentiable with respect to $\mu$ at 0, we have the first term in the bound to be $o(\lvert \mu\rvert)$. Given $(i)$ and $(ii)$, the conditions $(ii)$ and $(iii)$ in \Cref{prop1} are satisfied for $f_{\mu,\sigma}(y)$ with respect to $\sigma$. Given $(iii)$ and $(iv)$, by \cref{sigma_integrable}, we have $\dot{\xi}_{(0,1)_\sigma}$ is square integrable and $\mathbb{E}\dot \xi_{(0,\sigma)_\sigma}^2(y)\to \mathbb{E}\dot\xi_{(0,1)_\sigma}^2(y)$ as $\sigma\to 1$, which means $(iv)$ in \Cref{prop1} holds (condition $(i)$ in \Cref{prop1} is trivial). Hence, we obtain that $f_{\mu,\sigma}(y)$ is Hellinger differentiable with respect to $\sigma$ at 1. Then we have the second term in the bound to be $o(\lvert \sigma-1\rvert)$. Also, through $(v)$, we have
\begin{equation*}
	\begin{aligned}
		\lVert\xi_{\mu,\sigma}-\xi_{0,\sigma}-\xi_{\mu,1}+\xi_{0,1}\rVert&=\lVert\mu\dot\xi_{(0,\sigma)_\mu}+r_{0,\sigma}-\mu\dot\xi_{(0,1)_\mu}-r_{0,1}\rVert\\
		&\le \lVert\mu\dot\xi_{(0,\sigma)_\mu}-\mu\dot\xi_{(0,1)_\mu}\rVert + \lVert r_{0,\sigma}\rVert+ \lVert r_{0,1}\rVert\\
		&= \lVert\mu\dot\xi_{(0,\sigma)_\mu}-\mu\dot\xi_{(0,1)_\mu}\rVert+o(\lvert \mu\rvert)\\
		&=O(\lvert \mu(\sigma-1)\rvert)+o(\lvert \mu\rvert)\\
		&=o(\sqrt{\mu^2+(\sigma-1)^2}).
	\end{aligned}
\end{equation*}

Therefore, we obtain
\begin{equation*}
	\lVert r_{\mu,\sigma}(y)\rVert=o(\sqrt{\mu^2+(\sigma-1)^2}),
\end{equation*}
which completes the proof.
\end{proof}

\bibliographystyle{chicago}

\bibliography{bibliography}

\newpage
\section{Supplementary Materials -- Proofs of Corollaries 1 and 2}\label{supplementary materials}

\subsection{Proof of \Cref{cor1}}

Based on \Cref{th1}, we only need to show that \Cref{con1} and \Cref{con2} are satisfied for the error distribution location-scale family.

\subsubsection{Proof of \Cref{cor1} with ALD error distribution}
\begin{proof}

If the error distribution is ALD with location parameter $\eta$ and scale parameter $\sigma$, our target location family has the density
\begin{equation*} 
f_{\eta,\sigma}(y)=\frac{\tau(1-\tau)}{\sigma}\exp\Big[{-\rho_{\tau}\Big(\frac{y-\eta}{\sigma}\Big)}\Big]. 
\end{equation*}
We first show \Cref{con1} holds by showing $D(f_{0,1}||f_{\eta,\sigma})$ is twice differentiable with respect to $(\eta,\sigma)$ at $(0,1)$. Given $\tau$, we have

\begin{equation*}
\begin{aligned}
	D(f_{0,1}||f_{\eta,\sigma})=\frac{1}{\sigma}\mathbb{E}(\rho_\tau(Y-\eta)-\rho_\tau(Y)),
\end{aligned}
\end{equation*}
where $Y\sim f_{0,1}$. If $\eta>0$, by direct calculation, we have
\begin{equation*}
\begin{aligned}
\mathbb{E}[\rho_{\tau}(Y-\eta)]
&=\mathbb{E}(\tau(Y-\eta)I(Y>\eta)+(\tau-1)(Y-\eta)I(Y\le\eta))\\
&=\frac{\tau(1-\tau)}{\sigma}\Big[\frac{\sigma^2}{\tau^2}e^{-\frac{\tau}{\sigma}\eta}+\frac{\sigma}{\tau}\eta+\frac{\sigma^2}{1-\tau}-\frac{(1-\tau)\sigma^2}{\tau^2}\Big].\\
\end{aligned}
\end{equation*}

If $\eta\le 0$, similarly we have
\begin{equation*}
\begin{aligned}
\mathbb{E}[\rho_{\tau}(Y-\eta)]
&=\mathbb{E}(\tau(Y-\eta)I(Y>\eta)+(\tau-1)(Y-\eta)I(Y\le\eta))\\
&=\frac{\tau(1-\tau)}{\sigma}\Big[\frac{\sigma^2}{(1-\tau)^2}e^{\frac{1-\tau}{\sigma}\eta}-\frac{\sigma}{1-\tau}\eta+\frac{\sigma^2}{\tau}-\frac{\tau\sigma^2}{(1-\tau)^2}\Big].\\
\end{aligned}
\end{equation*}

In particular, 
\begin{equation*}
\begin{aligned}
	\mathbb{E}\rho_{\tau}(Y)
	&=\frac{\tau(1-\tau)}{\sigma}\Big(\frac{\sigma^2}{\tau}+\frac{\sigma^2}{1-\tau}\Big)=\sigma.
\end{aligned}
\end{equation*}

It is seen that the K-L divergence $D(f_{0,1}||f_{\eta,\sigma})$ is twice differentiable with respect to $(\eta,\sigma)$ at $(0,1)$. Hence, we have that \Cref{con1} is satisfied.

Next, we verify \Cref{con2}. We check the Hellinger differentiability of the error distribution family by applying \Cref{lm2}. First, $g(y)$ is absolutely continuous, which means \Cref{lm2}$(i)$ holds. Second, for $y\neq 0$, $g(y)$ has almost sure derivatives
\begin{equation*}
	\dot g(y)=\tau(1-\tau)\{-\tau\exp[-\tau yI(y>0)]+(1-\tau)\exp[(1-\tau)yI(y<\eta)]\},
\end{equation*}
which means \Cref{lm2}$(ii)$ is satisfied.

Also, for $y\neq \eta$, $\dot{\xi}_{(\eta,1)_\eta}^2(y)$ is
\begin{equation*}
\begin{aligned}
\dot{\xi}_{(\eta,1)_\eta}^2(y)
&=\frac{\tau(1-\tau)}{4}\{\tau^2\exp[-\tau (y-\eta)I(y>\eta)]+(\tau-1)^2\exp[(1-\tau)(y-\eta)I(y<\eta)]\},
\end{aligned}
\end{equation*}
and $\int \dot{\xi}_{(\eta,1)_\eta}^2(y)dy=\frac{\tau(1-\tau)}{4}$, we have the function $\dot\xi_{(\eta,1)_\eta}(y):=-\frac{1}{2}\dot g(y-\eta)I(g(y-\eta)>0)/\sqrt{g(y-\eta)}$ is squared integrable and $\mathbb{E}\dot{\xi}_{(\eta,1)_\eta}^2(y)\to \mathbb{E}\dot{\xi}_{(0,1)_\eta}^2(y)$ as $\eta\to 0$, which implies \Cref{lm2}$(iii)$. Moreover, since
\begin{equation*}
\begin{aligned}
\int &y^2\dot{\xi}^2_{(0,1)_\eta}(y)dy\\
&=\frac{\tau(1-\tau)}{4}\int\{\tau^2y^2\exp[-\tau yI(y>0)]+(\tau-1)^2y^2\exp[(1-\tau)yI(y<0)]\}dy<\infty,
\end{aligned}
\end{equation*}
we have \Cref{lm2}$(iv)$ holds. Finally, for each $y$, by Taylor expansion of the function $s_\sigma(y)=\dot\xi_{(0,\sigma)_\eta}(y)$ with respect to $\sigma$ at 1, we have
\begin{equation*}
	s_\sigma(y) = s_{\sigma=1}(y)+(\sigma-1)\dot s_{\sigma=1}(y)+o(\lvert \sigma-1\rvert),
\end{equation*}
where $\dot s_{\sigma=1}(y)$ is square integrable, we have 
\begin{equation*}
	\begin{aligned}
		\lVert \dot\xi_{(0,\sigma)_\eta}(y)-\dot\xi_{(0,1)_\eta}(y)\lVert_2=\lVert (\sigma-1)\dot s_{\sigma=1}(y)+o(\lvert \sigma-1\rvert)\lVert_2=O(\lvert \sigma-1\rvert).
	\end{aligned}
\end{equation*}

Hence, \Cref{lm2}$(v)$ is satisfied and we have that the location-scale family $\{\frac{1}{\sigma}f(\frac{x-\eta}{\sigma}),\eta\in\mathbb{R},\sigma\in(0,+\infty)\}$ is Hellinger differentiable with respect to $(\eta,\sigma)$ at $(0,1)$. Therefore, \Cref{con2} is satisfied.

\end{proof}

\subsubsection{Proof of \Cref{cor1} with ACDTG error distribution}
\begin{proof}
Although we have the shape parameter $\alpha\ge 0$, the following proof is valid only for $\alpha>0$. Fortunately, when $\alpha=0$, the ACDTG is the same as ALD, which has already been handled.

For the error distribution ACDTG with location parameter $\eta$ and scale parameter $\sigma$, our target location-scale family then has the density 
\begin{equation*}
f_{\eta,\sigma}(y)=\frac{\tau(1-\tau)}{\sigma\Gamma(\alpha+1,\alpha)}[\alpha+\rho_\tau(\frac{y-\eta}{\sigma})]^\alpha \exp\{-[\alpha+\rho_\tau(\frac{y-\eta}{\sigma})]\}.
\end{equation*}

For the K-L divergence, given $\tau$, we have
\begin{equation*}
\begin{aligned}
    D(f_{0,1}||f_{\eta,\sigma})=\alpha \mathbb{E}[\log(\alpha+\rho_\tau(Y))-\log(\alpha+\rho_\tau(\frac{Y-\eta}{\sigma}))]+\mathbb{E}[\rho_{\tau}(\frac{Y-\eta}{\sigma})-\rho_{\tau}(Y)]+\log\sigma,
\end{aligned}
\end{equation*}
where $Y\sim f_{0,1}$.

To show \Cref{con1} is satisfied, it suffice to show that $D(f_{0,1}||f_{\eta,\sigma})$ is finite and twice differentiable with respect to $(\eta,\sigma)$ at $(0,1)$. If $\eta>0$, we have
\begin{equation*}
	\begin{aligned}
		&\mathbb{E}\log(\alpha+\rho_\tau(\frac{Y-\eta}{\sigma}))\\
		&=\frac{\tau(1-\tau)}{\Gamma(\alpha+1,\alpha)}\Big[\int_{\eta}^\infty\log(\alpha+\tau (\frac{y-\eta}{\sigma}))(\alpha+\tau y)^{\alpha}e^{-(\alpha+\tau y)}dy\\
		&~~~~~~~~~~~~~~~~~~~+\int^{\eta}_0\log(\alpha+(\tau-1)(\frac{y-\eta}{\sigma}))(\alpha+\tau y)^{\alpha}e^{-(\alpha+\tau y)}dy\\
		&~~~~~~~~~~~~~~~~~~~+\int_{-\infty}^0\log(\alpha+(\tau-1)(\frac{y-\eta}{\sigma}))(\alpha+(\tau-1) y)^{\alpha}e^{-(\alpha+(\tau-1)y)}dy\Big],\\
	\end{aligned}
\end{equation*}
and
\begin{equation*}
	\begin{aligned}
		&\mathbb{E}\log(\alpha+\rho_\tau(Y))\\
		&=\frac{\tau(1-\tau)}{\Gamma(\alpha+1,\alpha)}\Big[\int_{0}^\infty\log(\alpha+\tau y)(\alpha+\tau y)^{\alpha}e^{-(\alpha+\tau y)}dy\\
		&~~~~~~~~~~~~~~~~~~~+\int_{-\infty}^0\log(\alpha+(\tau-1)y)(\alpha+(\tau-1) y)^{\alpha}e^{-(\alpha+(\tau-1)y)}dy\Big].\\
	\end{aligned}
\end{equation*}

Also, we have
\begin{equation*}
\begin{aligned}
&\mathbb{E}\rho_{\tau}(\frac{Y-\eta}{\sigma})\\
&=\mathbb{E}(\tau(\frac{Y-\eta}{\sigma})I(Y>\eta)+(\tau-1)(\frac{Y-\eta}{\sigma})I(Y\le\eta))\\
&=\frac{\tau(1-\tau)}{\Gamma(\alpha+1,\alpha)}\Big[\tau\int_{\eta}^\infty(\frac{y-\eta}{\sigma})(\alpha+\tau y)^\alpha e^{-(\alpha+\tau y)}dy-(1-\tau)\int_0^{\eta}(\frac{y-\eta}{\sigma})(\alpha+\tau y)^\alpha e^{-(\alpha+\tau y)}dy\\
&~~~~~~~~~~~~~~~~~~~~~~~-(1-\tau)\int_{-\infty}^0(\frac{y-\eta}{\sigma})(\alpha+(\tau-1) y)^\alpha e^{-(\alpha+(\tau-1) y)}dy\Big],
\end{aligned}
\end{equation*}
and
\begin{equation*}
\begin{aligned}
\mathbb{E}\rho_{\tau}(Y)
&=\mathbb{E}(\tau YI(Y>0)+(\tau-1)YI(Y\le 0))\\
&=\frac{\tau(1-\tau)}{\Gamma(\alpha+1,\alpha)}\Big[\tau\int_0^\infty y(\alpha+\tau y)^\alpha e^{-(\alpha+\tau y)}dy\\
&~~~~~~~~~~~~~-(1-\tau)\int_{-\infty}^0 y(\alpha+(\tau-1)y)^\alpha e^{-(\alpha+(\tau-1)y)}dy\Big].
\end{aligned}
\end{equation*}

Thus, the K-L divergence between $f_{0,1}$ and $f_{\eta,\sigma}$ is
\begin{equation*}
\begin{aligned}
	&\alpha \mathbb{E}[\log(\alpha+\rho_\tau(Y))-\log(\alpha+\rho_\tau(\frac{Y-\eta}{\sigma}))]+\mathbb{E}[\rho_{\tau}(\frac{Y-\eta}{\sigma})-\rho_{\tau}(Y)]+\log\sigma \\
	&=\alpha\frac{\tau(1-\tau)}{\Gamma(\alpha+1,\alpha)}\Big[(\int_{0}^\infty\log(\alpha+\tau y)(\alpha+\tau y)^{\alpha}e^{-(\alpha+\tau y)}dy\\
		&~~~~~~~~~~~~+\int_{-\infty}^0\log(\alpha+(\tau-1)y)(\alpha+(\tau-1) y)^{\alpha}e^{-(\alpha+(\tau-1)y)}dy\\
		&~~~~~~~~~~~~-\int_{\eta}^\infty\log(\alpha+\tau (\frac{y-\eta}{\sigma}))(\alpha+\tau y)^{\alpha}e^{-(\alpha+\tau y)}dy\\
		&~~~~~~~~~~~~-\int^{\eta}_0\log(\alpha+(\tau-1)(\frac{y-\eta}{\sigma}))(\alpha+\tau y)^{\alpha}e^{-(\alpha+\tau y)}dy\\
		&~~~~~~~~~~~~-\int_{-\infty}^0\log(\alpha+(\tau-1)(\frac{y-\eta}{\sigma}))(\alpha+(\tau-1) y)^{\alpha}e^{-(\alpha+(\tau-1)y)}dy\\
	&~~~~~~~~~~~~+\frac{\tau}{\alpha}\int_{\eta}^\infty(\frac{y-\eta}{\sigma})(\alpha+\tau y)^\alpha e^{-(\alpha+\tau y)}dy-\frac{1-\tau}{\alpha}\int_0^{\eta}(\frac{y-\eta}{\sigma})(\alpha+\tau y)^\alpha e^{-(\alpha+\tau y)}dy\\
    &~~~~~~~~~~~~-\frac{1-\tau}{\alpha}\int_{-\infty}^0(\frac{y-\eta}{\sigma})(\alpha+(\tau-1) y)^\alpha e^{-(\alpha+(\tau-1) y)}dy\\
    &~~~~~~~~~~~~-\frac{\tau}{\alpha}\int_0^\infty y(\alpha+\tau y)^\alpha e^{-(\alpha+\tau y)}dy+\frac{1-\tau}{\alpha}\int_{-\infty}^0 y(\alpha+(\tau-1)y)^\alpha e^{-(\alpha+(\tau-1)y)}dy\Big]+\log\sigma.
	\end{aligned}
\end{equation*}

It is straightforward to verify that the above K-L divergence is finite for each $(\eta,\sigma)$. By the Leibnitz integral rule, take the derivative of $D(f_{0,1}||f_{\eta,\sigma})$ with respect to $\eta$, we have 
\begin{equation*}
	\begin{aligned}
		&\frac{dD(f_{0,1}||f_{\eta,\sigma})}{d\eta}=\alpha\frac{\tau(1-\tau)}{\Gamma(\alpha+1,\alpha)}\Big[\int_{\eta}^\infty\frac{\tau}{\alpha\sigma+\tau(y-\eta)}(\alpha+\tau y)^{\alpha}e^{-(\alpha+\tau y)}dy\\
		&~~~~~~~~~~~~~~~~~~~~~~~~~~~~~~~-\int_{0}^\eta\frac{1-\tau}{\alpha\sigma+(\tau-1)(y-\eta)}(\alpha+\tau y)^{\alpha}e^{-(\alpha+\tau y)}dy\\
		&~~~~~~~~~~~~~~~~~~~~~~~~~~~~~~~-\int_{-\infty}^0\frac{1-\tau}{\alpha\sigma+(\tau-1)(y-\eta)}(\alpha+(\tau-1)y)^{\alpha}e^{-(\alpha+(\tau-1)y)}dy\\
		&~~~~~~~~~~~~~~~~~~~~~~~~~~~~~~~-\frac{1}{\alpha\sigma}\int_{\eta}^\infty\tau(\alpha+\tau y)^{\alpha}e^{-(\alpha+\tau y)}dy+\frac{1}{\alpha\sigma}\int_{0}^\eta(1-\tau)(\alpha+\tau y)^{\alpha}e^{-(\alpha+\tau y)}dy\\
		&~~~~~~~~~~~~~~~~~~~~~~~~~~~~~~~+\frac{1}{\alpha\sigma}\int_{-\infty}^0(1-\tau)(\alpha+(\tau-1)y)^{\alpha}e^{-(\alpha+(\tau-1)y)}dy\Big].
	\end{aligned}
\end{equation*}

Take the derivative of $\frac{dD(f_{0,1}||f_{\eta,\sigma})}{d\eta}$ with respect to $\eta$, we have
\begin{equation*}
	\begin{aligned}
		&\frac{d^2D(f_{0,1}||f_{\eta,\sigma})}{d\eta^2}=\alpha\frac{\tau(1-\tau)}{\Gamma(\alpha+1,\alpha)}\Big[\int_{\eta}^\infty\frac{\tau^2}{(\alpha\sigma+\tau(y-\eta))^2}(\alpha+\tau y)^{\alpha}e^{-(\alpha+\tau y)}dy\\
		&~~~~~~~~~~~~~~~~~~~~~~~~~~~~~~~~~~~~+\int_{0}^\eta\frac{(1-\tau)^2}{(\alpha\sigma+(\tau-1)(y-\eta))^2}(\alpha+\tau y)^{\alpha}e^{-(\alpha+\tau y)}dy\\
		&~~~~~~~~~~~~~~~~~~~~~~~~~~~~~~~~~~~~+\int_{-\infty}^0\frac{(1-\tau)^2}{(\alpha\sigma+(\tau-1)(y-\eta))^2}(\alpha+(\tau-1)y)^{\alpha}e^{-(\alpha+(\tau-1)y)}dy.
	\end{aligned}
\end{equation*}

Take the derivative of $\frac{dD(f_{0,1}||f_{\eta,\sigma})}{d\eta}$ with respect to $\sigma$, we have 
\begin{equation*}
	\begin{aligned}
		&\frac{d^2D(f_{0,1}||f_{\eta,\sigma})}{d\eta d\sigma}=\alpha\frac{\tau(1-\tau)}{\Gamma(\alpha+1,\alpha)}\Big[-\alpha\int_{\eta}^\infty\frac{\tau}{(\alpha\sigma+\tau(y-\eta))^2}(\alpha+\tau y)^{\alpha}e^{-(\alpha+\tau y)}dy\\
		&~~~~~~~~~~~~~~~~~~~~~~~~~~~~~~+\alpha\int_{0}^\eta\frac{1-\tau}{(\alpha\sigma+(\tau-1)(y-\eta))^2}(\alpha+\tau y)^{\alpha}e^{-(\alpha+\tau y)}dy\\
		&~~~~~~~~~~~~~~~~~~~~~~~~~~~~~~+\alpha\int_{-\infty}^0\frac{1-\tau}{(\alpha\sigma+(\tau-1)(y-\eta))^2}(\alpha+(\tau-1)y)^{\alpha}e^{-(\alpha+(\tau-1)y)}dy\\
		&~~~~~~~~~~~~~~~~~~~~~~~~~~~~~~+\frac{1}{\alpha\sigma^2}\int_{\eta}^\infty\tau(\alpha+\tau y)^{\alpha}e^{-(\alpha+\tau y)}dy-\frac{1}{\alpha\sigma^2}\int_{0}^\eta(1-\tau)(\alpha+\tau y)^{\alpha}e^{-(\alpha+\tau y)}dy\\
		&~~~~~~~~~~~~~~~~~~~~~~~~~~~~~~-\frac{1}{\alpha\sigma^2}\int_{-\infty}^0(1-\tau)(\alpha+(\tau-1)y)^{\alpha}e^{-(\alpha+(\tau-1)y)}dy\Big].
	\end{aligned}
\end{equation*}

Take the derivative of $D(f_{0,1}||f_{\eta,\sigma})$ with respect to $\sigma$, we have 
\begin{equation*}
	\begin{aligned}
		&\frac{dD(f_{0,1}||f_{\eta,\sigma})}{d\sigma}=\alpha\frac{\tau(1-\tau)}{\Gamma(\alpha+1,\alpha)}\Big[\frac{y-\eta}{\sigma}\int_{\eta}^\infty\frac{\tau}{\alpha\sigma+\tau(y-\eta)}(\alpha+\tau y)^{\alpha}e^{-(\alpha+\tau y)}dy\\
		&~~~~~~~~~~~~~~~~~~~~~~~~~~~~-\frac{y-\eta}{\sigma}\int_{0}^\eta\frac{1-\tau}{\alpha\sigma+(\tau-1)(y-\eta)}(\alpha+\tau y)^{\alpha}e^{-(\alpha+\tau y)}dy\\
		&~~~~~~~~~~~~~~~~~~~~~~~~~~~~-\frac{y-\eta}{\sigma}\int_{-\infty}^0\frac{1-\tau}{\alpha\sigma+(\tau-1)(y-\eta)}(\alpha+(\tau-1)y)^{\alpha}e^{-(\alpha+(\tau-1)y)}dy\\
		&~~~~~~~~~~~~~~~~~~~~~~~~~~~~-\frac{1}{\alpha\sigma^2}\int_{\eta}^\infty\tau(\alpha+\tau y)^{\alpha}e^{-(\alpha+\tau y)}dy+\frac{1}{\alpha\sigma^2}\int_{0}^\eta(1-\tau)(\alpha+\tau y)^{\alpha}e^{-(\alpha+\tau y)}dy\\
		&~~~~~~~~~~~~~~~~~~~~~~~~~~~~+\frac{1}{\alpha\sigma^2}\int_{-\infty}^0(1-\tau)(\alpha+(\tau-1)y)^{\alpha}e^{-(\alpha+(\tau-1)y)}dy\Big]+\frac{1}{\sigma}.
	\end{aligned}
\end{equation*}

Take the derivative of $\frac{dD(f_{0,1}||f_{\eta,\sigma})}{d\sigma}$ with respect to $\sigma$, we have 
\begin{equation*}
	\begin{aligned}
		&\frac{d^2D(f_{0,1}||f_{\eta,\sigma})}{d\sigma^2}=\alpha\frac{\tau(1-\tau)}{\Gamma(\alpha+1,\alpha)}\Big[-\frac{y-\eta}{\sigma^2}\int_{\eta}^\infty\frac{\tau}{\alpha\sigma+\tau(y-\eta)}(\alpha+\tau y)^{\alpha}e^{-(\alpha+\tau y)}dy\\
		&~~~~~~~~~~~~~~~~~~~~~~~~-\frac{\alpha(y-\eta)}{\sigma}\int_{\eta}^\infty\frac{\tau}{(\alpha\sigma+\tau(y-\eta))^2}(\alpha+\tau y)^{\alpha}e^{-(\alpha+\tau y)}dy\\
		&~~~~~~~~~~~~~~~~~~~~~~~~+\frac{y-\eta}{\sigma^2}\int_{0}^\eta\frac{1-\tau}{\alpha\sigma+(\tau-1)(y-\eta)}(\alpha+\tau y)^{\alpha}e^{-(\alpha+\tau y)}dy\\
		&~~~~~~~~~~~~~~~~~~~~~~~~+\frac{\alpha(y-\eta)}{\sigma}\int_{0}^\eta\frac{1-\tau}{(\alpha\sigma+(\tau-1)(y-\eta))^2}(\alpha+\tau y)^{\alpha}e^{-(\alpha+\tau y)}dy\\
		&~~~~~~~~~~~~~~~~~~~~~~~~+\frac{y-\eta}{\sigma^2}\int_{-\infty}^0\frac{1-\tau}{\alpha\sigma+(\tau-1)(y-\eta)}(\alpha+(\tau-1)y)^{\alpha}e^{-(\alpha+(\tau-1)y)}dy\\
		&~~~~~~~~~~~~~~~~~~~~~~~~+\frac{\alpha(y-\eta)}{\sigma}\int_{-\infty}^0\frac{1-\tau}{(\alpha\sigma+(\tau-1)(y-\eta))^2}(\alpha+(\tau-1)y)^{\alpha}e^{-(\alpha+(\tau-1)y)}dy\\
		&~~~~~~~~~~~~~~~~~~~~~~~~+\frac{2}{\alpha\sigma^3}\int_{\eta}^\infty\tau(\alpha+\tau y)^{\alpha}e^{-(\alpha+\tau y)}dy-\frac{2}{\alpha\sigma^3}\int_{0}^\eta(1-\tau)(\alpha+\tau y)^{\alpha}e^{-(\alpha+\tau y)}dy\\
		&~~~~~~~~~~~~~~~~~~~~~~~~-\frac{2}{\alpha\sigma^3}\int_{-\infty}^0(1-\tau)(\alpha+(\tau-1)y)^{\alpha}e^{-(\alpha+(\tau-1)y)}dy\Big]-\frac{1}{\sigma^2}.
	\end{aligned}
\end{equation*}

If $\eta< 0$, we obtain
\begin{equation*}
	\begin{aligned}
		&\mathbb{E}\log(\alpha+\rho_\tau(\frac{Y-\eta}{\sigma}))\\
		&=\frac{\tau(1-\tau)}{\Gamma(\alpha+1,\alpha)}\Big[\int_0^\infty\log(\alpha+\tau (\frac{y-\eta}{\sigma}))(\alpha+\tau y)^{\alpha}e^{-(\alpha+\tau y)}dy\\
		&~~~~~~~~~~~~~~~~~~~+\int_{\eta}^0\log(\alpha+\tau(\frac{y-\eta}{\sigma}))(\alpha+(\tau-1)y)^{\alpha}e^{-(\alpha+(\tau-1)y)}dy\\
		&~~~~~~~~~~~~~~~~~~~+\int_{-\infty}^\eta\log(\alpha+(\tau-1)(\frac{y-\eta}{\sigma}))(\alpha+(\tau-1) y)^{\alpha}e^{-(\alpha+(\tau-1)y)}dy\Big],\\
	\end{aligned}
\end{equation*}
and
\begin{equation*}
	\begin{aligned}
		&\mathbb{E}\log(\alpha+\rho_\tau(Y))\\
		&=\frac{\tau(1-\tau)}{\Gamma(\alpha+1,\alpha)}\Big[\int_{0}^\infty\log(\alpha+\tau y)(\alpha+\tau y)^{\alpha}e^{-(\alpha+\tau y)}dy\\
		&~~~~~~~~~~~~~~~~~~~+\int_{-\infty}^0\log(\alpha+(\tau-1)y)(\alpha+(\tau-1) y)^{\alpha}e^{-(\alpha+(\tau-1)y)}dy\Big].\\
	\end{aligned}
\end{equation*}

Also, we have
\begin{equation*}
\begin{aligned}
&\mathbb{E}\rho_{\tau}(\frac{Y-\eta}{\sigma})\\
&=\mathbb{E}(\tau(\frac{Y-\eta}{\sigma})I(Y>\eta)+(\tau-1)(\frac{Y-\eta}{\sigma})I(Y\le\eta))\\
&=\frac{\tau(1-\tau)}{\Gamma(\alpha+1,\alpha)}\Big[\tau\int_0^\infty(\frac{y-\eta}{\sigma})(\alpha+\tau y)^\alpha e^{-(\alpha+\tau y)}dy+\tau\int^0_{\eta}(\frac{y-\eta}{\sigma})(\alpha+(\tau-1)y)^\alpha e^{-(\alpha+(\tau-1)y)}dy\\
&~~~~~~~~~~~~~~~~~~~~~~~-(1-\tau)\int_{-\infty}^\eta(\frac{y-\eta}{\sigma})(\alpha+(\tau-1) y)^\alpha e^{-(\alpha+(\tau-1) y)}dy\Big],
\end{aligned}
\end{equation*}
and
\begin{equation*}
\begin{aligned}
\mathbb{E}\rho_{\tau}(Y)
&=\mathbb{E}(\tau YI(Y>0)+(\tau-1)YI(Y\le 0))\\
&=\frac{\tau(1-\tau)}{\Gamma(\alpha+1,\alpha)}\Big[\tau\int_0^\infty y(\alpha+\tau y)^\alpha e^{-(\alpha+\tau y)}dy\\
&~~~~~~~~~~~~~~~~~~~-(1-\tau)\int_{-\infty}^0 y(\alpha+(\tau-1)y)^\alpha e^{-(\alpha+(\tau-1)y)}dy\Big].
\end{aligned}
\end{equation*}

Thus, the K-L divergence between $f_{0,1}$ and $f_{\eta,\sigma}$ is
\begin{equation*}
\begin{aligned}
	&\alpha \mathbb{E}[\log(\alpha+\rho_\tau(Y))-\log(\alpha+\rho_\tau(\frac{Y-\eta}{\sigma}))]+\mathbb{E}[\rho_{\tau}(\frac{Y-\eta}{\sigma})-\rho_{\tau}(Y)]\\
	&=\alpha\frac{\tau(1-\tau)}{\Gamma(\alpha+1,\alpha)}\Big[(\int_{0}^\infty\log(\alpha+\tau y)(\alpha+\tau y)^{\alpha}e^{-(\alpha+\tau y)}dy\\
		&~~~~~~+\int_{-\infty}^0\log(\alpha+(\tau-1)y)(\alpha+(\tau-1) y)^{\alpha}e^{-(\alpha+(\tau-1)y)}dy\\
		&~~~~~~-\int_0^\infty\log(\alpha+\tau (\frac{y-\eta}{\sigma}))(\alpha+\tau y)^{\alpha}e^{-(\alpha+\tau y)}dy\\
		&~~~~~~-\int_{\eta}^0\log(\alpha+\tau(\frac{y-\eta}{\sigma}))(\alpha+(\tau-1)y)^{\alpha}e^{-(\alpha+(\tau-1)y)}dy\\
		&~~~~~~-\int_{-\infty}^\eta\log(\alpha+(\tau-1)(\frac{y-\eta}{\sigma}))(\alpha+(\tau-1) y)^{\alpha}e^{-(\alpha+(\tau-1)y)}dy\\
	&~~~~~~+\frac{\tau}{\alpha}\int_0^\infty(\frac{y-\eta}{\sigma})(\alpha+\tau y)^\alpha e^{-(\alpha+\tau y)}dy+\frac{\tau}{\alpha}\int^0_{\eta}(\frac{y-\eta}{\sigma})(\alpha+(\tau-1)y)^\alpha e^{-(\alpha+(\tau-1)y)}dy\\
    &~~~~~~-\frac{1-\tau}{\alpha}\int_{-\infty}^\eta(\frac{y-\eta}{\sigma})(\alpha+(\tau-1) y)^\alpha e^{-(\alpha+(\tau-1) y)}dy\\
    &~~~~~~-\frac{\tau}{\alpha}\int_0^\infty y(\alpha+\tau y)^\alpha e^{-(\alpha+\tau y)}dy+\frac{1-\tau}{\alpha}\int_{-\infty}^0 y(\alpha+(\tau-1)y)^\alpha e^{-(\alpha+(\tau-1)y)}dy\Big]+\log\sigma.
	\end{aligned}
\end{equation*}

Similarly as before, take the derivative of $D(f_{0,1}||f_{\eta,\sigma})$ with respect to $\eta$, we have
\begin{equation*}
	\begin{aligned}
		&\frac{dD(f_{0,1}||f_{\eta,\sigma})}{d\eta}=\alpha\frac{\tau(1-\tau)}{\Gamma(\alpha+1,\alpha)}\Big[\int_0^\infty\frac{\tau}{\alpha\sigma+\tau(y-\eta)}(\alpha+\tau y)^{\alpha}e^{-(\alpha+\tau y)}dy\\
		&~~~~~~~~~~~~~~~~~~~~~~~~~~+\int^{0}_\eta\frac{\tau}{\alpha\sigma+\tau(y-\eta)}(\alpha+(\tau-1)y)^{\alpha}e^{-(\alpha+(\tau-1)y)}dy\\
		&~~~~~~~~~~~~~~~~~~~~~~~~~~-\int_{-\infty}^\eta\frac{1-\tau}{\alpha\sigma+(\tau-1)(y-\eta)}(\alpha+(\tau-1)y)^{\alpha}e^{-(\alpha+(\tau-1)y)}dy\\
		&~~~~~~~~~~~~~~~~~~~~~~~~~~-\frac{1}{\alpha\sigma}\int_0^\infty\tau(\alpha+\tau y)^{\alpha}e^{-(\alpha+\tau y)}dy-\frac{1}{\alpha\sigma}\int^{0}_\eta\tau(\alpha+(\tau-1)y)^{\alpha}e^{-(\alpha+(\tau-1)y)}dy\\
		&~~~~~~~~~~~~~~~~~~~~~~~~~~+\frac{1}{\alpha\sigma}\int_{-\infty}^\eta(1-\tau)(\alpha+(\tau-1)y)^{\alpha}e^{-(\alpha+(\tau-1)y)}dy\Big].
	\end{aligned}
\end{equation*}

Take the derivative of $\frac{dD(f_{0,1}||f_{\eta,\sigma})}{d\eta}$ with respect to $\eta$, we have
\begin{equation*}
	\begin{aligned}
		&\frac{d^2D(f_0||f_{\eta})}{d\eta^2}=\alpha\frac{\tau(1-\tau)}{\Gamma(\alpha+1,\alpha)}\Big[\int_0^\infty\frac{\tau^2}{(\alpha\sigma+\tau(y-\eta))^2}(\alpha+\tau y)^{\alpha}e^{-(\alpha+\tau y)}dy\\
		&~~~~~~~~~~~~~~~~~~~~~~~~~~~~~~~~~~~~~~~~+\int^{0}_\eta\frac{\tau^2}{(\alpha\sigma+\tau(y-\eta))^2}(\alpha+(\tau-1)y)^{\alpha}e^{-(\alpha+(\tau-1)y)}dy\\
		&~~~~~~~~~~~~~~~~~~~~~~~~~~~~~~~~~~~~~~~~+\int_{-\infty}^\eta\frac{(1-\tau)^2}{(\alpha\sigma+(\tau-1)(y-\zeta))^2}(\alpha+(\tau-1)y)^{\alpha}e^{-(\alpha+(\tau-1)y)}dy.
	\end{aligned}
\end{equation*}

Take the derivative of $\frac{dD(f_{0,1}||f_{\eta,\sigma})}{d\eta}$ with respect to $\sigma$, we have 
\begin{equation*}
	\begin{aligned}
		&\frac{d^2D(f_{0,1}||f_{\eta,\sigma})}{d\eta d\sigma}=\alpha\frac{\tau(1-\tau)}{\Gamma(\alpha+1,\alpha)}\Big[-\alpha\int_0^\infty\frac{\tau}{\alpha\sigma+\tau(y-\eta)}(\alpha+\tau y)^{\alpha}e^{-(\alpha+\tau y)}dy\\
		&~~~~~~~~~~~~~~~~~~~~~~~~~~-\alpha\int^{0}_\eta\frac{\tau}{\alpha\sigma+\tau(y-\eta)}(\alpha+(\tau-1)y)^{\alpha}e^{-(\alpha+(\tau-1)y)}dy\\
		&~~~~~~~~~~~~~~~~~~~~~~~~~~+\alpha\int_{-\infty}^\eta\frac{1-\tau}{\alpha\sigma+(\tau-1)(y-\eta)}(\alpha+(\tau-1)y)^{\alpha}e^{-(\alpha+(\tau-1)y)}dy\\
		&~~~~~~~~~~~~~~~~~~~~~~~~~~+\frac{1}{\alpha\sigma^2}\int_0^\infty\tau(\alpha+\tau y)^{\alpha}e^{-(\alpha+\tau y)}dy+\frac{1}{\alpha\sigma^2}\int^{0}_\eta\tau(\alpha+(\tau-1)y)^{\alpha}e^{-(\alpha+(\tau-1)y)}dy\\
		&~~~~~~~~~~~~~~~~~~~~~~~~~~-\frac{1}{\alpha\sigma^2}\int_{-\infty}^\eta(1-\tau)(\alpha+(\tau-1)y)^{\alpha}e^{-(\alpha+(\tau-1)y)}dy\Big].
	\end{aligned}
\end{equation*}

Take the derivative of $D(f_{0,1}||f_{\eta,\sigma})$ with respect to $\sigma$, we have 
\begin{equation*}
	\begin{aligned}
		&\frac{dD(f_{0,1}||f_{\eta,\sigma})}{d\sigma}=\alpha\frac{\tau(1-\tau)}{\Gamma(\alpha+1,\alpha)}\Big[\frac{y-\eta}{\sigma}\int_0^\infty\frac{\tau}{\alpha\sigma+\tau(y-\eta)}(\alpha+\tau y)^{\alpha}e^{-(\alpha+\tau y)}dy\\
		&~~~~~~~~~~~~~~~~~~~~~~~~~+\frac{y-\eta}{\sigma}\int^{0}_\eta\frac{\tau}{\alpha\sigma+\tau(y-\eta)}(\alpha+(\tau-1)y)^{\alpha}e^{-(\alpha+(\tau-1)y)}dy\\
		&~~~~~~~~~~~~~~~~~~~~~~~~~-\frac{y-\eta}{\sigma}\int_{-\infty}^\eta\frac{1-\tau}{\alpha\sigma+(\tau-1)(y-\eta)}(\alpha+(\tau-1)y)^{\alpha}e^{-(\alpha+(\tau-1)y)}dy\\
		&~~~~~~~~~~~~~~~~~~~~~~~~~-\frac{1}{\alpha\sigma^2}\int_0^\infty\tau(\alpha+\tau y)^{\alpha}e^{-(\alpha+\tau y)}dy-\frac{1}{\alpha\sigma^2}\int^{0}_\eta\tau(\alpha+(\tau-1)y)^{\alpha}e^{-(\alpha+(\tau-1)y)}dy\\
		&~~~~~~~~~~~~~~~~~~~~~~~~~+\frac{1}{\alpha\sigma^2}\int_{-\infty}^\eta(1-\tau)(\alpha+(\tau-1)y)^{\alpha}e^{-(\alpha+(\tau-1)y)}dy\Big]-\frac{1}{\sigma}.
	\end{aligned}
\end{equation*}

Take the derivative of $\frac{dD(f_{0,1}||f_{\eta,\sigma})}{d\sigma}$ with respect to $\sigma$, we have 
\begin{equation*}
	\begin{aligned}
		&\frac{d^2D(f_{0,1}||f_{\eta,\sigma})}{d\sigma^2}=\alpha\frac{\tau(1-\tau)}{\Gamma(\alpha+1,\alpha)}\Big[-\frac{y-\eta}{\sigma^2}\int_0^\infty\frac{\tau}{\alpha\sigma+\tau(y-\eta)}(\alpha+\tau y)^{\alpha}e^{-(\alpha+\tau y)}dy\\
		&~~~~~~~~~~~~~~~~~~~~~~~~-\frac{\alpha(y-\eta)}{\sigma}\int_0^\infty\frac{\tau}{(\alpha\sigma+\tau(y-\eta))^2}(\alpha+\tau y)^{\alpha}e^{-(\alpha+\tau y)}dy\\
		&~~~~~~~~~~~~~~~~~~~~~~~~-\frac{y-\eta}{\sigma^2}\int^{0}_\eta\frac{\tau}{\alpha\sigma+\tau(y-\eta)}(\alpha+(\tau-1)y)^{\alpha}e^{-(\alpha+(\tau-1)y)}dy\\
		&~~~~~~~~~~~~~~~~~~~~~~~~-\frac{\alpha(y-\eta)}{\sigma}\int^{0}_\eta\frac{\tau}{(\alpha\sigma+\tau(y-\eta))^2}(\alpha+(\tau-1)y)^{\alpha}e^{-(\alpha+(\tau-1)y)}dy\\
		&~~~~~~~~~~~~~~~~~~~~~~~~+\frac{y-\eta}{\sigma^2}\int_{-\infty}^\eta\frac{1-\tau}{\alpha\sigma+(\tau-1)(y-\eta)}(\alpha+(\tau-1)y)^{\alpha}e^{-(\alpha+(\tau-1)y)}dy\\
		&~~~~~~~~~~~~~~~~~~~~~~~~+\frac{\alpha(y-\eta)}{\sigma}\int_{-\infty}^\eta\frac{1-\tau}{(\alpha\sigma+(\tau-1)(y-\eta))^2}(\alpha+(\tau-1)y)^{\alpha}e^{-(\alpha+(\tau-1)y)}dy\\
		&~~~~~~~~~~~~~~~~~~~~~~~~+\frac{2}{\alpha\sigma^3}\int_0^\infty\tau(\alpha+\tau y)^{\alpha}e^{-(\alpha+\tau y)}dy+\frac{2}{\alpha\sigma^3}\int^{0}_\eta\tau(\alpha+(\tau-1)y)^{\alpha}e^{-(\alpha+(\tau-1)y)}dy\\
		&~~~~~~~~~~~~~~~~~~~~~~~~-\frac{2}{\alpha\sigma^3}\int_{-\infty}^\eta(1-\tau)(\alpha+(\tau-1)y)^{\alpha}e^{-(\alpha+(\tau-1)y)}dy\Big]-\frac{1}{\sigma^2}.
	\end{aligned}
\end{equation*}

Since $\frac{d^2D(f_{0,1}||f_{\eta,\sigma})}{d\eta^2}$, $\frac{d^2D(f_{0,1}||f_{\eta,\sigma})}{d\eta d\sigma}$ and $\frac{d^2D(f_{0,1}||f_{\eta,\sigma})}{d\sigma^2}$ are well defined and continuous with respect to $(\eta,\sigma)$ near $(0,1)$, we have $D(f_{0,1}||f_{\eta,\sigma})$ is second differentiable with respect to $(\eta,\sigma)$ at $(0,1)$, which means \Cref{con1} is satisfied.

Next, we verify \Cref{con2}. By \Cref{lm2}, we check the Hellinger differentiability  of the error distribution family. First, $g(y)$ is absolutely continuous, which means \Cref{lm2}$(i)$ holds. Second, for $y\neq 0$, $g(y)$ has almost sure derivatives
\begin{equation*}
\begin{aligned}
\dot{g}(y)
&=\frac{\tau(1-\tau)}{\Gamma(\alpha+1,\alpha)}\big\{\alpha\tau[\alpha+\tau y]^{\alpha-1}\exp[-(\alpha+\tau y)]I(y>0)\\
&~~~~~~~~~~~~~~~~~~~~~~~~-\tau[\alpha+\tau y]^{\alpha}\exp[-(\alpha+\tau y)]I(y>0)\\
&~~~~~~~~~~~~~~~~~~~~~~~~+\alpha(\tau-1)[\alpha+(\tau-1)y]^{\alpha-1}\exp[-(\alpha+(\tau-1)y]I(y<0)\\
&~~~~~~~~~~~~~~~~~~~~~~~~-(\tau-1)[\alpha+(\tau-1)y]^{\alpha}\exp[-(\alpha+(\tau-1)y)]I(y<0)\big\},
\end{aligned}
\end{equation*}
which means \Cref{lm2}$(ii)$ is satisfied.

For $y\neq \eta$, $\dot{\xi}_{(\eta,1)_\eta}^2(y)$ is as following
\begin{equation*}
\begin{aligned}
\dot{\xi}_{(\eta,1)_\eta}^2(y)
&=\frac{\tau(1-\tau)}{\Gamma(\alpha+1,\alpha)}\big\{\frac{\alpha^2\tau^2}{4}[\alpha+\tau (y-\eta)]^{\alpha-2}\exp[-(\alpha+\tau (y-\eta))]I(y>\eta)\\
&~~~~~~~~+\frac{\tau^2}{4}[\alpha+\tau (y-\eta)]^\alpha \exp[-(\alpha+\tau (y-\eta)]I(y>\eta)\\
&~~~~~~~~-\frac{\alpha\tau^2}{2}[\alpha+\tau (y-\eta)]^{\alpha-1} \exp[-(\alpha+\tau (y-\eta))]I(y>\eta)\\
&~~~~~~~~+\frac{\alpha^2(\tau-1)^2}{4}[\alpha+(\tau-1) (y-\eta)]^{\alpha-2}\exp[-(\alpha+(\tau-1)(y-\eta))]I(y< \eta)\\
&~~~~~~~~+\frac{(\tau-1)^2}{4}[\alpha+(\tau-1)(y-\eta)]^\alpha \exp[-(\alpha+(\tau-1)(y-\eta))]I(y< \eta)\\
&~~~~~~~~-\frac{\alpha(\tau-1)^2}{2}[\alpha+(\tau-1)(y-\eta)]^{\alpha-1} \exp[-(\alpha+(\tau-1)(y-\eta))]I(y< \eta)\big\}.
\end{aligned}
\end{equation*}

Through a direct calculation, we can get the upper bound of integral of $\dot{\xi}_{(\eta,1)_\eta}^2(y)$
\begin{equation*}
\begin{aligned}
&\int\dot{\xi}_{(\eta,1)_\eta}^2(y)dy\\
&\le \frac{\tau(1-\tau)}{\Gamma(\alpha+1,\alpha)}\big\{\max\{\frac{\alpha^2\tau^2}{4},\frac{\alpha^2(\tau-1)^2}{4}\}\int_{-\infty}^\infty[\alpha+\rho(y-\eta)]^{\alpha-2}\exp[-(\alpha+\rho(y-\eta))]dy\\
&~~~~~~~~~~~~~~~~~~~~~~~+\max\{\frac{\tau^2}{4},\frac{(\tau-1)^2}{4}\}\int_{-\infty}^\infty[\alpha+\rho(y-\eta)]^\alpha \exp[-(\alpha+\rho(y-\eta)]dy\\
&~~~~~~~~~~~~~~~~~~~~~~~-\min\{\frac{\alpha\tau^2}{2},\frac{\alpha(\tau-1)^2}{2}\}\int_{-\infty}^\infty[\alpha+\rho(y-\eta)]^{\alpha-1} \exp[-(\alpha+\rho(y-\eta))]dy\big\}.\\
\end{aligned}
\end{equation*}

Also, we have 
\begin{equation*}
	\begin{aligned}
		\int_{-\infty}^\infty[\alpha+\rho(y-\eta)]^\alpha \exp[-(\alpha+\rho(y-\eta))]dy=\frac{\Gamma(\alpha+1,\alpha)}{\tau(1-\tau)},
	\end{aligned}
\end{equation*}
\begin{equation*}
	\begin{aligned}
		\int_{-\infty}^\infty[\alpha+\rho(y-\eta)]^{\alpha-1} \exp[-(\alpha+\rho(y-\eta))]dy\le \frac{\Gamma(\alpha+1,\alpha)}{\alpha\tau(1-\tau)},
	\end{aligned}
\end{equation*}
\begin{equation*}
	\begin{aligned}
		\int_{-\infty}^\infty[\alpha+\rho(y-\eta)]^{\alpha-2} \exp[-(\alpha+\rho(y-\eta))]dy\le \frac{\Gamma(\alpha+1,\alpha)}{\alpha^2\tau(1-\tau)}.
	\end{aligned}
\end{equation*}

Therefore, we obtain
\begin{equation*}
\begin{aligned}
\int\dot{\xi}_{(\eta,1)_\eta}^2(y)dy\le \max\{\frac{\tau^2}{2},\frac{(\tau-1)^2}{2}\},
\end{aligned}
\end{equation*} 
which is a constant for given $\alpha$ and $\tau$. Hence, we know that the function $\dot\xi_{(\eta,1)_\eta}(y):=-\frac{1}{2}\dot g(y-\eta)I(g(y-\eta)>0)/\sqrt{g(y-\eta)}$ is square integrable and $\mathbb{E}\dot{\xi}_{(\eta,1)_\eta}^2(y)\to \mathbb{E}\dot{\xi}_{(0,1)_\eta}^2(y)$ as $\eta\to 0$, which implies \Cref{lm2}$(iii)$. Moreover, since
\begin{equation*}
\begin{aligned}
\int y^2\dot{\xi}^2_{(0,1)_\eta}(y)dy
&=\frac{\tau(1-\tau)}{\Gamma(\alpha+1,\alpha)}\int y^2\big\{\frac{\alpha^2\tau^2}{4}(\alpha+\tau y)^{\alpha-2}\exp[-(\alpha+\tau y)]I(y>0)\\
&~~~~~~~+\frac{\tau^2}{4}(\alpha+\tau y)^\alpha \exp[-(\alpha+\tau y)]I(y>0)\\
&~~~~~~~-\frac{\alpha\tau^2}{2}(\alpha+\tau y)^{\alpha-1} \exp[-(\alpha+\tau y)]I(y>0)\\
&~~~~~~~+\frac{\alpha^2(\tau-1)^2}{4}[\alpha+(\tau-1)y]^{\alpha-2}\exp[-(\alpha+(\tau-1)y)]I(y<0)\\
&~~~~~~~+\frac{(\tau-1)^2}{4}[\alpha+(\tau-1)y]^\alpha \exp[-(\alpha+(\tau-1)y)]I(y<0)\\
&~~~~~~~-\frac{\alpha(\tau-1)^2}{2}[\alpha+(\tau-1)y]^{\alpha-1} \exp[-(\alpha+(\tau-1)y)]I(y<0)\big\}dy<\infty,
\end{aligned}
\end{equation*} 
we have that \Cref{lm2}$(iv)$ holds. Finally, for each $y$, by Taylor expansion of the function $s_\sigma(y)=\dot\xi_{(0,\sigma)_\eta}(y)$ with respect to $\sigma$ at 1, we have
\begin{equation*}
	s_\sigma(y) = s_{\sigma=1}(y)+(\sigma-1)\dot s_{\sigma=1}(y)+o(\lvert \sigma-1\rvert),
\end{equation*}
where $\dot s_{\sigma=1}(y)$ is square integrable, we have 
\begin{equation*}
	\begin{aligned}
		\lVert \dot\xi_{(0,\sigma)_\eta}(y)-\dot\xi_{(0,1)_\eta}(y)\lVert_2=\lVert (\sigma-1)\dot s_{\sigma=1}(y)+o(\lvert \sigma-1\rvert)\lVert_2=O(\lvert \sigma-1\rvert).
	\end{aligned}
\end{equation*}

Hence, \Cref{lm2}$(v)$ is satisfied and we obtain that the location-scale family $\{\frac{1}{\sigma}f(\frac{x-\eta}{\sigma}),\eta\in\mathbb{R},\sigma\in(0,+\infty)\}$ is Hellinger differentiable with respect to $(\eta,\sigma)$ at $(0,1)$. Therefore, \Cref{con2} is satisfied.

\end{proof}

\subsubsection{Proof of \Cref{cor1} with CNL error distribution}
\begin{proof}
Our location-scale family has density
\begin{equation*}
	f_{\eta,\sigma}(y)=\frac{1-\tau}{\beta\sigma}\exp[-\frac{1}{\beta\sigma}(y-\eta)]I(y> \eta)+\frac{2\tau}{\sqrt{2\pi}\sigma}\exp[-\frac{(y-\eta)^2}{2\sigma^2}]I(y\le \eta),-\infty<y<\infty.
\end{equation*}
We first show \Cref{con1} holds. If $\eta>0$, we have
\begin{equation*}
	\begin{aligned}
		D(f_{0,1}||f_{\eta,\sigma})&=\mathbb{E}\log \frac{f_{0,1}}{f_{\eta,\sigma}}\\
		&=\mathbb{E}\log \frac{\frac{1-\tau}{\beta}e^{-\frac{1}{\beta}y}I(y>0)+\frac{2\tau}{\sqrt{2\pi}}e^{-\frac{y^2}{2}}I(y\le 0)}{\frac{1-\tau}{\beta\sigma}e^{-\frac{1}{\beta\sigma}(y-\eta)}I(y>\eta)+\frac{2\tau}{\sqrt{2\pi}\sigma}e^{-\frac{(y-\eta)^2}{2\sigma^2}}I(y\le \eta)}\\
		&=\frac{1-\tau}{\beta}\int_{\eta}^\infty (\frac{1}{\beta\sigma}(y-\eta)-\frac{1}{\beta}y+\log\sigma)e^{-\frac{1}{\beta}y}dy\\
		&~~~~+\frac{1-\tau}{\beta}\int_0^\eta[\log((1-\tau)\frac{\sqrt{2\pi}\sigma}{2\beta \tau})-\frac{1}{\beta}y+\frac{(y-\eta)^2}{2\sigma^2}]e^{-\frac{1}{\beta}y}dy\\
		&~~~~+\frac{2\tau}{\sqrt{2\pi}}\int_{-\infty}^0(\log\sigma-\frac{y^2}{2}+\frac{(y-\eta)^2}{2\sigma^2})e^{-\frac{y^2}{2}}dy.
	\end{aligned}
\end{equation*}

By Leibniz rule, we have
\begin{equation*}
	\begin{aligned}
		\frac{dD(f_{0,1}||f_{\eta,\sigma})}{d\eta}&=-\frac{1-\tau}{\beta\sigma} e^{-\frac{1}{\beta}\eta}+\frac{1-\tau}{\beta}\log(\frac{(1-\tau)\sqrt{2\pi}}{2\beta \tau})e^{-\frac{1}{\beta}\eta}\\
		&~~~~~+\frac{1-\tau}{\beta\sigma^2}\int_0^\eta(\eta-y)e^{-\frac{1}{\beta}y}dy+\frac{2\tau}{\sqrt{2\pi}\sigma^2}\int_{-\infty}^0(\eta-y)e^{-\frac{y^2}{2}}dy,
	\end{aligned}
\end{equation*}
\begin{equation*}
	\begin{aligned}
		\frac{d^2D(f_0||f_\eta)}{d\eta^2}&=\frac{1-\tau}{\beta^2\sigma}e^{-\frac{1}{\beta}\eta}-\frac{1-\tau}{\beta^2}\log(\frac{(1-\tau)\sqrt{2\pi}}{2\beta \tau})e^{-\frac{1}{\beta}\eta}\\
		&~~~~+\frac{1-\tau}{\beta\sigma^2}\int_0^\eta e^{-\frac{1}{\beta}y}dy+\frac{2\tau}{\sqrt{2\pi}\sigma^2}\int_{-\infty}^0 e^{-\frac{y^2}{2}}dy,
	\end{aligned}
\end{equation*}

\begin{equation*}
	\begin{aligned}
		\frac{d^2D(f_{0,1}||f_{\eta,\sigma})}{d\eta d\sigma}&=\frac{1-\tau}{\beta\sigma^2} e^{-\frac{1}{\beta}\eta}-\frac{2(1-\tau)}{\beta\sigma^3}\int_0^\eta(\eta-y)e^{-\frac{1}{\beta}y}dy-\frac{4\tau}{\sqrt{2\pi}\sigma^3}\int_{-\infty}^0(\eta-y)e^{-\frac{y^2}{2}}dy,
	\end{aligned}
\end{equation*}

\begin{equation*}
	\begin{aligned}
		\frac{dD(f_{0,1}||f_{\eta,\sigma})}{d\sigma}&=\frac{1-\tau}{\beta}\int_\eta^\infty[\frac{1}{\sigma}-\frac{1}{\beta\sigma^2}(y-\eta)]e^{-\frac{1}{\beta}y}dy\\
		&+\frac{1-\tau}{\beta}\int_0^\eta[\frac{1}{\sigma}-\frac{(y-\eta)^2}{\sigma^3}]e^{-\frac{1}{\beta}y}dy\\
		&+\frac{2\tau}{\sqrt{2\pi}}\int_{-\infty}^0[\frac{1}{\sigma}-\frac{(y-\eta)^2}{\sigma^3}]e^{-\frac{y^2}{2}}dy,
	\end{aligned}
\end{equation*}
and
\begin{equation*}
	\begin{aligned}
		\frac{d^2D(f_{0,1}||f_{\eta,\sigma})}{d\sigma^2}&=\frac{1-\tau}{\beta}\int_\eta^\infty[-\frac{1}{\sigma^2}+\frac{2}{\beta\sigma^3}(y-\eta)]e^{-\frac{1}{\beta}y}dy\\
		&+\frac{1-\tau}{\beta}\int_0^\eta[-\frac{1}{\sigma^2}+\frac{3(y-\eta)^2}{\sigma^4}]e^{-\frac{1}{\beta}y}dy\\
		&+\frac{2\tau}{\sqrt{2\pi}}\int_{-\infty}^0[-\frac{1}{\sigma^2}+\frac{3(y-\eta)^2}{\sigma^4}]e^{-\frac{y^2}{2}}dy.
	\end{aligned}
\end{equation*}

Hence, if $\eta>0$, take the second derivative with respect to $(\eta,\sigma)$ at $(0,1)$ and the Hessian matrix at $(\eta,\sigma)=(0,1)$ is
 \[
  \left( {\begin{array}{cc}
   \frac{1-\tau}{\beta^2}-\frac{1-\tau}{\beta^2}\log\frac{(1-\tau)\sqrt{2\pi}}{2\beta \tau}+\tau & \frac{1-\tau}{\beta}+\frac{4\tau}{\sqrt{2\pi}} \\
   \frac{1-\tau}{\beta}+\frac{4\tau}{\sqrt{2\pi}} & 1+\tau \\
  \end{array} } \right)
\]

If $\eta<0$, we have
\begin{equation*}
	\begin{aligned}
		D(f_{0,1}||f_{\eta,\sigma})&=\mathbb{E}\log \frac{f_{0,1}}{f_{\eta,\sigma}}\\
		&=\mathbb{E}\log \frac{\frac{1-\tau}{\beta}e^{-\frac{1}{\beta}y}I(y>0)+\frac{2\tau}{\sqrt{2\pi}}e^{-\frac{y^2}{2}}I(y\le 0)}{\frac{1-\tau}{\beta\sigma}e^{-\frac{1}{\beta\sigma}(y-\eta)}I(y>\eta)+\frac{2\tau}{\sqrt{2\pi}\sigma}e^{-\frac{(y-\eta)^2}{2\sigma^2}}I(y\le \eta)}\\
		&=\frac{1-\tau}{\beta}\int_{0}^\infty [\log\sigma-\frac{1}{\beta}y+\frac{1}{\beta\sigma}(y-\eta)] e^{-\frac{1}{\beta}y}dy\\
		&~~~~+\frac{2\tau}{\sqrt{2\pi}}\int_\eta^0[\log(\frac{2\tau\beta\sigma}{(1-\tau)\sqrt{2\pi}})-\frac{y^2}{2}+\frac{1}{\beta\sigma}(y-\eta)]e^{-\frac{y^2}{2}}dy\\
		&~~~~+\frac{2\tau}{\sqrt{2\pi}}\int_{-\infty}^\eta[\log\sigma-\frac{y^2}{2}+\frac{(y-\eta)^2}{2\sigma^2}]e^{-\frac{y^2}{2}}dy.
	\end{aligned}
\end{equation*}

By Leibniz rule, we have
\begin{equation*}
	\begin{aligned}
		\frac{dD(f_{0,1}||f_{\eta,\sigma})}{d\eta}&=-\frac{1-\tau}{\beta\sigma}-\log(\frac{2\tau\beta}{(1-\tau)\sqrt{2\pi}})\frac{2\tau}{\sqrt{2\pi}}e^{-\frac{\eta^2}{2}}-\frac{2\tau}{\sqrt{2\pi}\beta\sigma}\int_{\eta}^0e^{-\frac{y^2}{2}}dy\\
		&~~~~~+\frac{2\tau}{\sqrt{2\pi}\sigma^2}\int_{-\infty}^\eta(\eta-y)e^{-\frac{y^2}{2}}dy,
	\end{aligned}
\end{equation*}
\begin{equation*}
	\begin{aligned}
		\frac{d^2D(f_{0,1}||f_{\eta,\sigma})}{d\eta^2}&=\log(\frac{2\tau\beta}{(1-\tau)\sqrt{2\pi}})\frac{2\tau\eta}{\sqrt{2\pi}}e^{-\frac{\eta^2}{2}}+\frac{2\tau}{\sqrt{2\pi}\beta\sigma}e^{-\frac{\eta^2}{2}}\\
		&~~~~~+\frac{2\tau}{\sqrt{2\pi}\sigma^2}\int_{-\infty}^\eta e^{-\frac{y^2}{2}}dy,
	\end{aligned}
\end{equation*}

\begin{equation*}
	\begin{aligned}
		\frac{d^2D(f_{0,1}||f_{\eta,\sigma})}{d\eta d\sigma}&=\frac{1-\tau}{\beta\sigma^2}+\frac{2\tau}{\sqrt{2\pi}\beta\sigma^2}\int_{\eta}^0e^{-\frac{y^2}{2}}dy\\
		&~~~~~-\frac{4\tau}{\sqrt{2\pi}\sigma^3}\int_{-\infty}^\eta(\eta-y)e^{-\frac{y^2}{2}}dy,
	\end{aligned}
\end{equation*}

\begin{equation*}
	\begin{aligned}
		\frac{dD(f_{0,1}||f_{\eta,\sigma})}{d\sigma}&=\frac{1-\tau}{\beta}\int_{0}^\infty [\frac{1}{\sigma}-\frac{1}{\beta\sigma^2}(y-\eta)] e^{-\frac{1}{\beta}y}dy\\
		&~~~~+\frac{2\tau}{\sqrt{2\pi}}\int_\eta^0[\frac{1}{\sigma}-\frac{1}{\beta\sigma^2}(y-\eta)]e^{-\frac{y^2}{2}}dy\\
		&~~~~+\frac{2\tau}{\sqrt{2\pi}}\int_{-\infty}^\eta[\frac{1}{\sigma}-\frac{(y-\eta)^2}{\sigma^3}]e^{-\frac{y^2}{2}}dy,
	\end{aligned}
\end{equation*}
and
\begin{equation*}
	\begin{aligned}
		\frac{d^2D(f_{0,1}||f_{\eta,\sigma})}{d\sigma^2}&=\frac{1-\tau}{\beta}\int_{0}^\infty [-\frac{1}{\sigma^2}+\frac{2}{\beta\sigma^3}(y-\eta)] e^{-\frac{1}{\beta}y}dy\\
		&~~~~+\frac{2\tau}{\sqrt{2\pi}}\int_\eta^0[-\frac{1}{\sigma^2}+\frac{2}{\beta\sigma^3}(y-\eta)]e^{-\frac{y^2}{2}}dy\\
		&~~~~+\frac{2\tau}{\sqrt{2\pi}}\int_{-\infty}^\eta[-\frac{1}{\sigma^2}+\frac{3(y-\eta)^2}{\sigma^4}]e^{-\frac{y^2}{2}}dy.
	\end{aligned}
\end{equation*}

Hence, if $\eta<0$, take the second derivative with respect to $(\eta,\sigma)$ at $(0,1)$ and the Hessian matrix at $(\eta,\sigma)=(0,1)$ is
 \[
  \left( {\begin{array}{cc}
   \frac{2\tau}{\sqrt{2\pi}\beta}+\tau & \frac{1-\tau}{\beta}+\frac{4\tau}{\sqrt{2\pi}} \\
    \frac{1-\tau}{\beta} +\frac{4\tau}{\sqrt{2\pi}} & 1+\tau \\
  \end{array} } \right)
\]

If $\frac{2\tau}{\sqrt{2\pi}}=\frac{1-\tau}{\beta}$, we have the Hessian matrix for $\eta>0$ and $\eta<0$ are equivalent. Also, since $\frac{d^2D(f_{0,1}||f_{\eta,\sigma})}{d\eta^2}$, $\frac{d^2D(f_{0,1}||f_{\eta,\sigma})}{d\eta d\sigma}$ and $\frac{d^2D(f_{0,1}||f_{\eta,\sigma})}{d\sigma^2}$ are well defined and continuous with respect to $(\eta,\sigma)$ near $(0,1)$, we have $D(f_{0,1}||f_{\eta,\sigma})$ is twice differentiable with respect to $(\eta,\sigma)$ at $(0,1)$. Hence, \Cref{con1} is satisfied.

Next, we verify \Cref{con2}. By \Cref{lm2}, we check the Hellinger differentiability  of the error distribution. First, $g(y)$ is absolutely continuous, which implies \Cref{lm2}$(i)$ Second, for $y\neq 0$, $g(y)$ has almost sure derivatives
\begin{equation*}
	\dot g(y)=-\frac{1-\tau}{\beta^2}\exp[-\frac{1}{\beta}yI(y>0)]-\frac{2\tau y}{\sqrt{2\pi}}\exp[-\frac{y^2}{2}I(y<0)],
\end{equation*}
which means \Cref{lm2}$(ii)$ holds.

Moreover, for $y\neq \eta$, $\dot{\xi}_{(\eta,1)_\eta}^2(y)$ is
\begin{equation*}
\begin{aligned}
\dot{\xi}_{(\eta,1)_\eta}^2(y)
&=\frac{1-\tau}{4\beta^3}\exp[-\frac{1}{\beta}(y-\eta)I(y>\eta)]+\frac{\tau(y-\eta)^2}{2\sqrt{2\pi}}\exp[-\frac{(y-\eta)^2}{2}I(y<\eta)],
\end{aligned}
\end{equation*}
and $\int \dot{\xi}_{(\eta,1)_\eta}^2(y)dy=\frac{1-\tau}{4\beta^2}+\frac{\tau}{4}$, we have that the function $\dot\xi_{(\eta,1)_\eta}(y):=-\frac{1}{2}\dot g(y-\eta)I(g(y-\eta)>0)/\sqrt{g(y-\eta)}$ is squared integrable and $\mathbb{E}\dot{\xi}_{(\eta,1)_\eta}^2(y)\to \mathbb{E}\dot{\xi}_{(0,1)_\eta}^2(y)$ as $\eta\to 0$, which means \Cref{lm2}$(iii)$ is satisfied. Moreover, since
\begin{equation*}
	\begin{aligned}
		\int y^2\dot\xi_{(0,1)_\eta}^2(y)dy=\frac{1-\tau}{4\beta^3}\int y^2\exp[-\frac{1}{\beta}yI(y>0)]dy+\frac{\tau}{2\sqrt{2\pi}}\int y^4\exp[-\frac{y^2}{2}I(y<\eta)]dy<\infty,
	\end{aligned}
\end{equation*}
we have \Cref{lm2}$(iv)$ holds. Finally, for each $y$, by Taylor expansion of the function $s_\sigma(y)=\dot\xi_{(0,\sigma)_\eta}(y)$ with respect to $\sigma$ at 1, we have
\begin{equation*}
	s_\sigma(y) = s_{\sigma=1}(y)+(\sigma-1)\dot s_{\sigma=1}(y)+o(\lvert \sigma-1\rvert),
\end{equation*}
where $\dot s_{\sigma=1}(y)$ is square integrable, we obtain 
\begin{equation*}
	\begin{aligned}
		\lVert \dot\xi_{(0,\sigma)_\eta}(y)-\dot\xi_{(0,1)_\eta}(y)\lVert_2=\lVert (\sigma-1)\dot s_{\sigma=1}(y)+o(\lvert \sigma-1\rvert)\lVert_2=O(\lvert \sigma-1\rvert).
	\end{aligned}
\end{equation*}

Hence, \Cref{lm2}$(v)$ is satisfied and the location-scale family $\{\frac{1}{\sigma}f(\frac{x-\eta}{\sigma}),\eta\in \mathbb{R},\sigma\in (0,+\infty)\}$ is Hellinger differentiable with respect to $(\eta,\sigma)$ at $(0,1)$. Therefore, \Cref{con2} is satisfied.
\end{proof}

\subsubsection{Proof of \Cref{cor1} with Cauchy error distribution}
\begin{proof}
	We first show \Cref{con1} holds by showing $D(f_{0,1}||f_{\eta,\sigma})$ is twice differentiable with respect to $(\eta,\sigma)$ at $(0,1)$. By \cite{Chyzak}, we have
	\begin{equation*}
		D(f_{0,1}||f_{\eta,\sigma})=\log\Big[\frac{(1+\sigma)^2+\eta^2}{4\sigma}\Big].
	\end{equation*}
	It is seen that the K-L divergence $D(f_{0,1}||f_{\eta,\sigma})$ is twice differentiable with respect to $(\eta,\sigma)$ at $(0,1)$. Hence, we have that \Cref{con1} is satisfied.\\
	Next, by \Cref{lm2}, we verify \Cref{con2}. It is obvious that Cauchy density is absolutely continuous and differentiable, hence \Cref{lm2}$(i)$ and \Cref{lm2}$(ii)$ hold. Also, since 
	\begin{equation*}
		\int \dot\xi_{(\eta,1)_\eta}^2(y)dy=\frac{1}{\pi}\int \frac{y^2}{(1+y^2)^3}dy=\frac{1}{\pi}[(\frac{1}{8}\arctan y +\frac{y^3-y}{8(y^2+1)^2})|^{\infty}_{-\infty}]=\frac{1}{8},
	\end{equation*}
	and 
	\begin{equation*}
		\int y^2\dot\xi_{(\eta,1)_\eta}^2(y)dy=\frac{1}{\pi^2}\int \frac{y^4}{(1+y^2)^3}dy=\frac{1}{\pi^2}[(\frac{3}{8}\arctan y -\frac{5y^3+3y}{8(y^2+1)^2})|^{\infty}_{-\infty}]=\frac{3}{8\pi},
	\end{equation*}
	we have that \Cref{lm2}$(iii)$ and \Cref{lm2}$(iv)$ are satisfied. Finally, for each $y$, by Taylor expansion of the function $s_{\sigma}(y)=\dot \xi_{(0,\sigma)_\eta}(y)$ at $\sigma=1$, we have
	\begin{equation*}
		s_\sigma(y)=s_{\sigma=1}(y)+(\sigma-1)\dot s_{\sigma=1}(y)+o(\lvert \sigma-1\rvert),
	\end{equation*}
	where $\dot s_{\sigma=1}(y)$ is square integrable, we have
	\begin{equation*}
		\lVert \dot \xi_{(0,\sigma)_\eta}(y)- \dot \xi_{(0,1)_\eta}(y)\rVert_2=\lVert (\sigma-1)\dot s_{\sigma=1}(y)+o(\lvert \sigma-1\rvert)\rVert=O(\lvert \sigma-1\rvert),
	\end{equation*}
	which means \Cref{lm2}$(v)$ holds. Therefore, we have verified \Cref{con2}.
\end{proof}

\subsubsection{Proof of \Cref{cor2}}
\begin{proof}
	Our location-scale family has density
	\begin{equation*}
		f_{\mu,\sigma}(y)=\frac{2}{\sqrt{\pi}\sigma}\frac{\sqrt{\phi(1-\phi)}}{\sqrt{\phi}+\sqrt{1-\phi}}\exp[-\frac{\phi}{\sigma^2}(y-\eta)^2I(y\ge \eta)-\frac{1-\phi}{\sigma^2}(y-\eta)^2I(y<\eta)], -\infty<y<\infty,
	\end{equation*}
	
	We first show \Cref{con1} holds. If $\eta > 0$, we have
	\begin{equation*}
		\begin{aligned}
			D(f_{0,1}||f_{\mu,\sigma})&=\mathbb{E}\log\frac{f_{0,1}}{f_{\mu,\sigma}}\\
			&=\mathbb{E}\log\frac{\sigma \exp [-\phi y^2I(y\ge 0)-(1-\phi)y^2I(y<0)]}{\exp [-\frac{\phi}{\sigma^2}(y-\eta)^2I(y\ge \eta)-\frac{1-\phi}{\sigma^2}(y-\eta)^2I(y< \eta)]}\\
			&=\log\sigma +\frac{2}{\sqrt{\pi}}\frac{\sqrt{\phi(1-\phi)}}{\sqrt{\phi}+\sqrt{1-\phi}}\{\int_\eta^\infty [\frac{\phi}{\sigma^2}(y-\eta)^2-\phi y^2]e^{-\phi y^2}dy\\
			&~~~~~+\int_0^\eta [\frac{1-\phi}{\sigma^2}(y-\eta)^2-\phi y^2]e^{-\phi y^2}dy\\
			&~~~~~+\int_{-\infty}^0 [\frac{1-\phi}{\sigma^2}(y-\eta)^2-(1-\phi) y^2]e^{-(1-\phi) y^2}dy\}.
		\end{aligned}
	\end{equation*}

	By Leibniz rule, we have
	\begin{equation*}
		\begin{aligned}
			\frac{dD(f_{0,1}||f_{\mu,\sigma})}{d\eta}&= \frac{2}{\sqrt{\pi}}\frac{\sqrt{\phi(1-\phi)}}{\sqrt{\phi}+\sqrt{1-\phi}}\{\int_\eta^\infty [-\frac{2\phi}{\sigma^2}(y-\eta)]e^{-\phi y^2}dy\\
			&~~~~~+\int^\eta_0 [-\frac{2(1-\phi)}{\sigma^2}(y-\eta)]e^{-\phi y^2}dy+\int_{-\infty}^0 [-\frac{2(1-\phi)}{\sigma^2}(y-\eta)]e^{-(1-\phi) y^2}dy\},
		\end{aligned}
	\end{equation*}
	\begin{equation*}
		\begin{aligned}
			\frac{d^2D(f_{0,1}||f_{\mu,\sigma})}{d\eta^2}&= \frac{2}{\sqrt{\pi}}\frac{\sqrt{\phi(1-\phi)}}{\sqrt{\phi}+\sqrt{1-\phi}}\{\int_\eta^\infty \frac{2\phi}{\sigma^2}e^{-\phi y^2}dy\\
			&~~~~~+\int^\eta_0 \frac{2(1-\phi)}{\sigma^2}e^{-\phi y^2}dy+\int_{-\infty}^0 \frac{2(1-\phi)}{\sigma^2}e^{-(1-\phi) y^2}dy\},
		\end{aligned}
	\end{equation*}
	\begin{equation*}
		\begin{aligned}
			\frac{d^2D(f_{0,1}||f_{\mu,\sigma})}{d\eta d\sigma}&= \frac{2}{\sqrt{\pi}}\frac{\sqrt{\phi(1-\phi)}}{\sqrt{\phi}+\sqrt{1-\phi}}\{\int_\eta^\infty [\frac{4\phi}{\sigma^3}(y-\eta)]e^{-\phi y^2}dy\\
			&~~~~~+\int^\eta_0 [\frac{4(1-\phi)}{\sigma^3}(y-\eta)]e^{-\phi y^2}dy+\int_{-\infty}^0 [\frac{4(1-\phi)}{\sigma^3}(y-\eta)]e^{-(1-\phi) y^2}dy\},
		\end{aligned}
	\end{equation*}
	\begin{equation*}
		\begin{aligned}
			\frac{dD(f_{0,1}||f_{\mu,\sigma})}{d\sigma}&=\frac{1}{\sigma}+\frac{2}{\sqrt{\pi}}\frac{\sqrt{\phi(1-\phi)}}{\sqrt{\phi}+\sqrt{1-\phi}}\{\int_\eta^\infty [-\frac{2\phi}{\sigma^3}(y-\eta)^2]e^{-\phi y^2}dy\\
			&~~~~~+\int^\eta_0 [-\frac{2(1-\phi)}{\sigma^3}(y-\eta)^2]e^{-\phi y^2}dy+\int_{-\infty}^0 [-\frac{2(1-\phi)}{\sigma^3}(y-\eta)^2]e^{-(1-\phi) y^2}dy\},
		\end{aligned}
	\end{equation*}
	and
	\begin{equation*}
		\begin{aligned}
			\frac{d^2D(f_{0,1}||f_{\mu,\sigma})}{d\sigma^2}&=-\frac{1}{\sigma^2}+\frac{2}{\sqrt{\pi}}\frac{\sqrt{\phi(1-\phi)}}{\sqrt{\phi}+\sqrt{1-\phi}}\{\int_\eta^\infty [\frac{6\phi}{\sigma^4}(y-\eta)^2]e^{-\phi y^2}dy\\
			&~~~~~+\int_0^\eta [\frac{6(1-\phi)}{\sigma^4}(y-\eta)^2]e^{-\phi y^2}dy+\int_{-\infty}^0 [\frac{6(1-\phi)}{\sigma^4}(y-\eta)^2]e^{-(1-\phi) y^2}dy\}.
		\end{aligned}
	\end{equation*}

	Hence, if $\eta>0$, take the second derivative with respect to $(\eta,\sigma)$ at (0,1) and we get the Hessian matrix at $(\eta,\sigma)=(0,1)$ is 
  \[
  \left( {\begin{array}{cc}
    2\sqrt{\phi(1-\phi)} & \frac{4}{\sqrt{\pi}}\frac{\sqrt{\phi(1-\phi)}}{\sqrt{\phi}+\sqrt{1-\phi}} \\
    \frac{4}{\sqrt{\pi}}\frac{\sqrt{\phi(1-\phi)}}{\sqrt{\phi}+\sqrt{1-\phi}} & 2 \\
  \end{array} } \right)
  \]

   If $\eta<0$, we have
   \begin{equation*}
		\begin{aligned}
			D(f_{0,1}||f_{\mu,\sigma})&=\mathbb{E}\log\frac{f_{0,1}}{f_{\mu,\sigma}}\\
			&=\mathbb{E}\log\frac{\sigma \exp [-\phi y^2I(y\ge 0)-(1-\phi)y^2I(y<0)]}{\exp [-\frac{\phi}{\sigma^2}(y-\eta)^2I(y\ge \eta)-\frac{1-\phi}{\sigma^2}(y-\eta)^2I(y< \eta)]}\\
			&=\log\sigma +\frac{2}{\sqrt{\pi}}\frac{\sqrt{\phi(1-\phi)}}{\sqrt{\phi}+\sqrt{1-\phi}}\{\int_0^\infty [\frac{\phi}{\sigma^2}(y-\eta)^2-\phi y^2]e^{-\phi y^2}dy\\
			&~~~~~+\int^0_\eta [\frac{\phi}{\sigma^2}(y-\eta)^2-(1-\phi) y^2]e^{-(1-\phi) y^2}dy\\
			&~~~~~+\int_{-\infty}^\eta [\frac{1-\phi}{\sigma^2}(y-\eta)^2-(1-\phi) y^2]e^{-(1-\phi) y^2}dy\}.
		\end{aligned}
	\end{equation*}

	By Leibniz rule, we have
	\begin{equation*}
		\begin{aligned}
			\frac{dD(f_{0,1}||f_{\mu,\sigma})}{d\eta}&= \frac{2}{\sqrt{\pi}}\frac{\sqrt{\phi(1-\phi)}}{\sqrt{\phi}+\sqrt{1-\phi}}\{\int_0^\infty [-\frac{2\phi}{\sigma^2}(y-\eta)]e^{-\phi y^2}dy\\
			&~~~~~+\int_\eta^0 [-\frac{2\phi}{\sigma^2}(y-\eta)]e^{-(1-\phi) y^2}dy+\int_{-\infty}^\eta [-\frac{2(1-\phi)}{\sigma^2}(y-\eta)]e^{-(1-\phi) y^2}dy\},
		\end{aligned}
	\end{equation*}
	\begin{equation*}
		\begin{aligned}
			\frac{d^2D(f_{0,1}||f_{\mu,\sigma})}{d\eta^2}&= \frac{2}{\sqrt{\pi}}\frac{\sqrt{\phi(1-\phi)}}{\sqrt{\phi}+\sqrt{1-\phi}}\{\int_0^\infty \frac{2\phi}{\sigma^2}e^{-\phi y^2}dy\\
			&~~~~~+\int_\eta^0 \frac{2\phi}{\sigma^2}e^{-(1-\phi) y^2}dy+\int_{-\infty}^\eta \frac{2(1-\phi)}{\sigma^2}e^{-(1-\phi) y^2}dy\},
		\end{aligned}
	\end{equation*}
	\begin{equation*}
		\begin{aligned}
			\frac{d^2D(f_{0,1}||f_{\mu,\sigma})}{d\eta d\sigma}&= \frac{2}{\sqrt{\pi}}\frac{\sqrt{\phi(1-\phi)}}{\sqrt{\phi}+\sqrt{1-\phi}}\{\int_0^\infty [\frac{4\phi}{\sigma^3}(y-\eta)]e^{-\phi y^2}dy\\
			&~~~~~+\int_\eta^0 [\frac{4\phi}{\sigma^3}(y-\eta)]e^{-(1-\phi) y^2}dy+\int_{-\infty}^\eta [\frac{4(1-\phi)}{\sigma^3}(y-\eta)]e^{-(1-\phi) y^2}dy\},
		\end{aligned}
	\end{equation*}
	\begin{equation*}
		\begin{aligned}
			\frac{dD(f_{0,1}||f_{\mu,\sigma})}{d\sigma}&=\frac{1}{\sigma}+\frac{2}{\sqrt{\pi}}\frac{\sqrt{\phi(1-\phi)}}{\sqrt{\phi}+\sqrt{1-\phi}}\{\int_0^\infty [-\frac{2\phi}{\sigma^3}(y-\eta)^2]e^{-\phi y^2}dy\\
			&~~~~~+\int_\eta^0 [-\frac{2\phi}{\sigma^3}(y-\eta)^2]e^{-(1-\phi) y^2}dy+\int_{-\infty}^\eta [-\frac{2(1-\phi)}{\sigma^3}(y-\eta)^2]e^{-(1-\phi) y^2}dy\},
		\end{aligned}
	\end{equation*}
	and
	\begin{equation*}
		\begin{aligned}
			\frac{d^2D(f_{0,1}||f_{\mu,\sigma})}{d\sigma^2}&=-\frac{1}{\sigma^2}+\frac{2}{\sqrt{\pi}}\frac{\sqrt{\phi(1-\phi)}}{\sqrt{\phi}+\sqrt{1-\phi}}\{\int_0^\infty [\frac{6\phi}{\sigma^4}(y-\eta)^2]e^{-\phi y^2}dy\\
			&~~~~~+\int_\eta^0 [\frac{6\phi}{\sigma^4}(y-\eta)^2]e^{-(1-\phi) y^2}dy+\int_{-\infty}^\eta [\frac{6(1-\phi)}{\sigma^4}(y-\eta)^2]e^{-(1-\phi) y^2}dy\}.
		\end{aligned}
	\end{equation*}

	Hence, if $\eta>0$, take the second derivative with respect to $(\eta,\sigma)$ at (0,1) and we get the Hessian matrix at $(\eta,\sigma)=(0,1)$ is 
  \[
  \left( {\begin{array}{cc}
    2\sqrt{\phi(1-\phi)} & \frac{4}{\sqrt{\pi}}\frac{\sqrt{\phi(1-\phi)}}{\sqrt{\phi}+\sqrt{1-\phi}} \\
    \frac{4}{\sqrt{\pi}}\frac{\sqrt{\phi(1-\phi)}}{\sqrt{\phi}+\sqrt{1-\phi}} & 2 \\
  \end{array} } \right)
  \]
  
We have the Hessian matrix for $\eta>0$ and $\eta<0$ are equivalent. Also, since $\frac{d^2D(f_{0,1}||f_{\eta,\sigma})}{d\eta^2}$, $\frac{d^2D(f_{0,1}||f_{\eta,\sigma})}{d\eta d\sigma}$ and $\frac{d^2D(f_{0,1}||f_{\eta,\sigma})}{d\sigma^2}$ are well defined and continuous with respect to $(\eta,\sigma)$ near $(0,1)$, we have $D(f_{0,1}||f_{\eta,\sigma})$ is twice differentiable with respect to $(\eta,\sigma)$ at $(0,1)$. Hence, \Cref{con1} is satisfied.

Next, we verify \Cref{con2}. By \Cref{lm2}, we check the Hellinger differentiability  of the error distribution. First, $g(y)$ is absolutely continuous, which means \Cref{lm2}$(i)$ is satisfied. Second, $g(y)$ has almost sure derivatives
\begin{equation*}
	\dot g(y)=-\frac{4}{\sqrt{\pi}}\frac{\sqrt{\phi(1-\phi)}}{\sqrt{\phi}+\sqrt{1-\phi}}[\phi ye^{-\phi y^2}I(y\ge 0)+(1-\phi)ye^{-(1-\phi)y^2}I(y<0)],
\end{equation*}
which means \Cref{lm2}$(ii)$ holds.

Moreover, for $y\neq \eta$, $\dot{\xi}_{(\eta,1)_\eta}^2(y)$ is
\begin{equation*}
\begin{aligned}
\dot{\xi}_{(\eta,1)_\eta}^2(y)
=\frac{2}{\sqrt{\pi}}\frac{\sqrt{\phi(1-\phi)}}{\sqrt{\phi}+\sqrt{1-\phi}}[\phi (y-\eta)e^{-\phi (y-\eta)^2}I(y\ge \eta)+(1-\phi)(y-\eta)e^{-(1-\phi)(y-\eta)^2}I(y<\eta)],
\end{aligned}
\end{equation*}
and $\int \dot{\xi}_{(\eta,1)_\eta}^2(y)dy=\frac{2}{\sqrt{\pi}}\frac{\sqrt{\phi(1-\phi)}}{\sqrt{\phi}+\sqrt{1-\phi}}$, we have that the function $\dot\xi_{(\eta,1)_\eta}(y):=-\frac{1}{2}\dot g(y-\eta)I(g(y-\eta)>0)/\sqrt{g(y-\eta)}$ is squared integrable and $\mathbb{E}\dot{\xi}_{(\eta,1)_\eta}^2(y)\to \mathbb{E}\dot{\xi}_{(0,1)_\eta}^2(y)$ as $\eta\to 0$, which implies \Cref{lm2}$(iii)$. Moreover, since
\begin{equation*}
	\begin{aligned}
		\int y^2\dot\xi_{(0,1)_\eta}^2(y)dy=\frac{2}{\sqrt{\pi}}\frac{\sqrt{\phi(1-\phi)}}{\sqrt{\phi}+\sqrt{1-\phi}}\int[\phi y^3e^{-\phi y^2}I(y\ge 0)+(1-\phi)y^3e^{-(1-\phi)y^2}I(y<0)]dy<\infty,
	\end{aligned}
\end{equation*}
we have \Cref{lm2}$(iv)$ holds. Finally, for each $y$, by Taylor expansion of the function $s_\sigma(y)=\dot\xi_{(0,\sigma)_\eta}(y)$ with respect to $\sigma$ at 1, we have
\begin{equation*}
	s_\sigma(y) = s_{\sigma=1}(y)+(\sigma-1)\dot s_{\sigma=1}(y)+o(\lvert \sigma-1\rvert),
\end{equation*}
where $\dot s_{\sigma=1}(y)$ is square integrable, we have 
\begin{equation*}
	\begin{aligned}
		\lVert \dot\xi_{(0,\sigma)_\eta}(y)-\dot\xi_{(0,1)_\eta}(y)\lVert_2=\lVert (\sigma-1)\dot s_{\sigma=1}(y)+o(\lvert \sigma-1\rvert)\lVert_2=O(\lvert \sigma-1\rvert).
	\end{aligned}
\end{equation*}

Hence, \Cref{lm2}$(v)$ holds and the location-scale family $\{\frac{1}{\sigma}f(\frac{x-\eta}{\sigma}),\eta\in \mathbb{R},\sigma\in (0,+\infty)\}$ is Hellinger differentiable with respect to $(\eta,\sigma)$ at $(0,1)$. Therefore, \Cref{con2} is satisfied.

\end{proof}

%

\end{document}